\documentclass[onefignum,onetabnum]{siamart251216}

\usepackage{amsopn}

\usepackage{soul}
\usepackage[normalem]{ulem}
\usepackage{amsfonts}
\usepackage{graphicx}
\usepackage{epstopdf}
\usepackage{algorithmic}
\ifpdf
  \DeclareGraphicsExtensions{.eps,.pdf,.png,.jpg}
\else
  \DeclareGraphicsExtensions{.eps}
\fi

\newsiamremark{remark}{Remark}

\headers{Reformulation of Fluid Problems on Closed Hypersurfaces}{R. H. Nochetto and M. Shakipov}

\title{Fluid Problems Posed on Closed Hypersurfaces: Circumventing the Saddle-Point Structure\thanks{Submitted to the editors DATE. \funding{Both authors were partially supported by the NSF grant DMS-2512392.}}}

\author{Ricardo H. Nochetto\thanks{Department of Mathematics and Institute for Physical Science and Technology, University of Maryland, College Park, MD 20742 USA (\email{rhn@umd.edu}).}
\and Mansur Shakipov\thanks{Department of Mathematics, University of Houston, Houston, TX 77204 USA
  (\email{mshakipov@uh.edu}).}}

\makeatletter
\newcommand*{\addFileDependency}[1]{
  \typeout{(#1)}
  \@addtofilelist{#1}
  \IfFileExists{#1}{}{\typeout{No file #1.}}
}
\makeatother

\usepackage{hyperref}
\usepackage{mathtools, amsmath, amssymb, bm, centernot, cancel, scalerel, dsfont}
\usepackage{mathrsfs}
\usepackage{enumitem}
\setlist[itemize]{leftmargin=*}
\DeclareMathAlphabet\mathbfcal{OMS}{cmsy}{b}{n}
\DeclareMathAlphabet\mathscrbf{OMS}{mdugm}{b}{n}

    \usepackage{accents}

    \newcommand{\pa}{\partial}

    \renewcommand{\bf}{\mathbf{f}}
    \newcommand{\bg}{\mathbf{g}}

    \newcommand{\bn}{\mathbf{n}}

    \newcommand{\bu}{\mathbf{u}}
    \newcommand{\bv}{\mathbf{v}}
    \newcommand{\bw}{\mathbf{w}}

    \newcommand{\bA}{\mathbf{A}}
    
    \newcommand{\bC}{\mathbf{C}}
    
    \newcommand{\bH}{\mathbf{H}}
    \newcommand{\bI}{\mathbf{I}}
    \newcommand{\bL}{\mathbf{L}}
    
    \newcommand{\bP}{\mathbf{P}}
    
    \newcommand{\bR}{\mathbf{R}}
    \newcommand{\bW}{\mathbf{W}}

    \newcommand{\bzero}{\mathbf{0}}

    \newcommand{\sV}{\mathscr{V}}
    
    \newcommand{\sX}{\mathscr{X}}
    \newcommand{\sY}{\mathscr{Y}}
    \newcommand{\sZ}{\mathscr{Z}}

    \newcommand{\cA}{\mathcal{A}}
    \newcommand{\cB}{\mathcal{B}}
    
    \newcommand{\cD}{\mathcal{D}}

    \newcommand{\cK}{\mathcal{K}}

    \newcommand{\cP}{\mathcal{P}}

    \newcommand{\cV}{\mathcal{V}}
    \newcommand{\cW}{\mathcal{W}}

    \newcommand{\bcV}{\boldsymbol{\mathcal{V}}}
    \newcommand{\bcW}{\boldsymbol{\mathcal{W}}}

    \newcommand{\bbR}{\mathbb{R}}
    \newcommand{\bbZ}{\mathbb{Z}}

    \newcommand{\dH}{\mathds{H}}
    \newcommand{\dL}{\mathds{L}}

    \renewcommand{\d}{\mathrm{d}}
    \newcommand{\dt}{\mathrm{d}t}
    \newcommand{\ds}{\mathrm{d}s}

    \newcommand{\embeds}{\hookrightarrow}

    \newcommand{\G}{\Gamma}

    \newcommand{\nablaG}{\nabla_\Gamma}
    \newcommand{\DG}{D_\Gamma}
    \newcommand{\divG}{\mathrm{div}_\Gamma}
    \newcommand{\DeltaG}{\Delta_\Gamma}
    \newcommand{\DeltaS}{\Delta_S}

    \DeclareMathOperator{\tr}{tr}

\ifpdf
\hypersetup{
  pdftitle={Fluid Problems Posed on Closed Hypersurfaces: Circumventing the Saddle-Point Structure},
  pdfauthor={R. H. Nochetto, M. Shakipov}
}
\fi

\begin{document}

\maketitle

\begin{abstract}
    For a $d$-dimensional closed hypersurface in $\mathbb{R}^{d+1}$, we reformulate the (stationary) tangent Stokes, transient tangent Stokes, and transient tangent Navier--Stokes equations by replacing the divergence equation with a pressure Poisson equation, thereby circumventing their saddle-point structure. The pressure Poisson equation hinges on a commutator identity between the divergence and the surface diffusion operator, whose lower-order term depends on the Ricci curvature, and requires no boundary conditions for the pressure. We show that the reformulated operator is the composition of a linear, block-triangular data map with the primal operator. Consequently, each reformulation is well-posed if and only if the primal problem is, and solving it with arbitrary data amounts to solving the primal problem with a modified divergence datum. For hypersurfaces of class $C^2$ and data of the pressure Poisson equation in $H^{-1}$, the tangent Stokes reformulations can be posed in $H^1$-based spaces, which paves the way for conforming discretizations that do not require a discrete inf-sup condition.
\end{abstract}



\section{Introduction}
    Recent years have seen an increasing interest in the modeling of fluid flow confined to thin curved domains, with applications ranging from biological membranes \cite{ArroyoDeSimone2009, EdwardsBrennerWasan1991} to active nematics \cite{doostmohammadi_active_2018, keber2014topology}. The modeling of these phenomena reduces to the (Navier--)Stokes equations posed on a stationary or moving surface embedded in $\bbR^3$ \cite{Scriven1960, ArroyoDeSimone2009, KobaLiuGiga2017, JankuhnOlshanskiiReusken2018}. As in the Euclidean case, the formulation of these equations in the primitive variables, velocity and pressure, has a saddle-point structure. Stable discretizations thus generally require a \textit{discrete inf-sup condition} \cite{Brezzi1974, BoffiBrezziFortin2013}, which restricts the choice of finite element spaces.
 
    In this paper, we show that this saddle-point structure can be circumvented when the fluid is confined to a closed hypersurface. The key point is that a closed hypersurface has no boundary, so the pressure can be determined by a pressure Poisson equation that needs no boundary condition. Throughout, $\G$ is a stationary, connected, and closed (i.e., compact and without boundary) hypersurface of dimension $d \geq 2$ in $\bbR^{d+1}$. We study the (stationary) tangent Stokes (TS), transient tangent Stokes (TTS), and transient tangent Navier--Stokes (TTNS) equations posed on $\G$. In unified form, these equations read as follows: given data $(\bf_0, \bf, f)$, find a velocity-pressure pair $(\bu, u)$ such that
    \begin{equation} \label{eq:tts-ttns}
        \begin{cases}
            \alpha_1 \pa_t \bu + \alpha_2 \bu \cdot \nablaG \bu - 2 \mu \DeltaS \bu + \nablaG u \kern-0.72em & = \bf, \\
            \divG \bu \kern-0.72em & = f, \\
            \alpha_1 \bu(0) \kern-0.72em & = \alpha_1 \bf_0.
        \end{cases}
    \end{equation}
    Here, the constant $\mu > 0$ is the surface viscosity, and $\alpha_1, \alpha_2 \in \{0, 1\}$ are indicators of transience and the convective term $\bu \cdot \nablaG \bu := (\nablaG \bu) \bu$, respectively. We denote the tangential gradient by $\nablaG$, the tangential divergence by $\divG$, and the Laplace--Beltrami operator by $\DeltaG$. The surface viscous operator $\DeltaS$ applies the tangential divergence to the symmetric part of the covariant gradient and projects the result onto the tangent space. This operator is the physically relevant choice for surface fluids \cite{KobaLiuGiga2017, JankuhnOlshanskiiReusken2018}. Section~\ref{ch:prelims;se:gen-notation;subse:geometry} gives the precise definitions. We use boldface for vector-valued quantities and normal font for their scalar counterparts; in particular, the pressure is denoted by $u$. We also stress that the divergence datum $f$ need not vanish, a fact that will be instrumental in the analysis.
 
    This work is motivated by a numerical method for TS introduced in \cite{NochettoShakipov2025}, which rests on an equation for the pressure. We formally derive the analogous equation for all three problems by applying $-\divG$ to the momentum equation in \eqref{eq:tts-ttns} and substituting $\divG \bu = f$. The resulting equation is relatively simple for two reasons. First, the tangential divergence commutes with the time derivative because $\G$ is stationary. Second, the tangential divergence commutes with $\DeltaS$ up to the lower-order curvature term $\divG (\bR \bu)$, where $\bR$ is the \textit{Ricci curvature map} of $\G$; see Lemma~\ref{lem:stokes-iden} (decomposition of $\DeltaS$). For $d = 2$, i.e., for surfaces in $\bbR^3$, the Ricci curvature map acts on tangential fields as multiplication by the Gauss curvature. The outcome is the \textit{pressure Poisson equation}
    \begin{equation} \label{eq:intro-pressure-poi}
        - \DeltaG u - \alpha_2 \divG (\bu \cdot \nablaG \bu) + 2 \mu \, \divG (\bR \bu) = - \divG \bf + \alpha_1 \pa_t f - 2 \mu \DeltaG f,
    \end{equation}
    which requires no boundary condition for the pressure because $\G$ is closed. Section~\ref{subsec:reform-of-stokes-type-pbs} contains the derivation.
 
    We call the system \eqref{eq:tts-ttns} the \textit{primal problem}. Replacing its divergence equation with the pressure Poisson equation \eqref{eq:intro-pressure-poi} yields what we call the \textit{reformulation}. In the reformulation, the pressure solves an explicit second-order equation, namely a perturbed Laplace--Beltrami problem, rather than acting as a Lagrange multiplier for the divergence equation. The reformulation thus \textit{circumvents the saddle-point structure}: its conforming discretizations do not require a discrete inf-sup condition. In \cite{NochettoShakipov2025}, this reformulation was derived and discretized for TS with a modified viscous operator.
 
    Every solution of the primal problem also solves the reformulation, with the right-hand side of the pressure Poisson equation as datum. A more delicate question is whether the reformulation remains solvable and well-behaved for arbitrary data. Such data arise naturally, for instance through perturbations or discretization, and the well-posedness of the reformulation with arbitrary data is crucial for its numerical analysis \cite{NochettoShakipov2025}. We therefore study the reformulation with \textit{general} data $(\bg_0, \bg, g)$:
    \begin{equation} \label{eq:tts-ttns-reform-gen-data}
        \begin{cases}
            \alpha_1 \pa_t \bu + \alpha_2 \bu \cdot \nablaG \bu - 2 \mu \DeltaS \bu + \nablaG u \kern-0.72em & = \bg, \\
            - \DeltaG u - \alpha_2 \divG (\bu \cdot \nablaG \bu) + 2 \mu \, \divG (\bR \bu) \kern-0.72em & = g, \\
            \alpha_1 \bu(0) \kern-0.72em & = \alpha_1 \bg_0.
        \end{cases}
    \end{equation}
    \textit{Our main result is that the reformulation \eqref{eq:tts-ttns-reform-gen-data} is well-posed if and only if the primal problem \eqref{eq:tts-ttns} is. Moreover, solving the reformulation with arbitrary data amounts to solving the primal problem with a modified divergence datum.}
 
    The proof rests on a simple algebraic structure. The pressure Poisson equation is a linear combination of the equations of the primal problem. Indeed, one applies $-\divG$ to both sides of the momentum equation, applies $\alpha_1 \pa_t - 2 \mu \DeltaG$ to both sides of the divergence equation, and adds the results. The reformulation thus arises from the primal problem by a row operation, much like a step of Gaussian elimination. We write the primal problem as $\cA x = y$ and the reformulation as $\cB x = z$. A linear \textit{data map} $\cD$ records the row operation. The data map leaves the initial and momentum data unchanged and replaces the divergence datum $f$ by the right-hand side of the pressure Poisson equation. Applied to the left-hand side $\cA x$, the data map yields $\cB x$; that is, the operators factor as $\cB = \cD \circ \cA$. This factorization rests on a cancellation: the term $(\alpha_1 \pa_t - 2 \mu \DeltaG) \divG \bu$ arises in both summands of the row operation with opposite signs. Recall that the tangential divergence commutes with $\pa_t$ and, up to the curvature term, with $\DeltaS$. After the cancellation, the velocity enters the pressure Poisson equation only through the curvature term and the convective term. Moreover, the data map is block triangular. Inverting the data map therefore amounts to solving a single Poisson problem ($\alpha_1 = 0$) or heat problem ($\alpha_1 = 1$) for the divergence datum. Consequently, the operator $\cB$ is a homeomorphism if and only if $\cA$ is, and the reformulation loses no information; see Proposition~\ref{prop:transf-of-well-pos} (transfer of well-posedness) and Lemma~\ref{lem:data-maps} (data maps). Since the data map does not involve the convective term, the nonlinear TTNS problem requires no additional work at this level.
 
    We now list our main results and briefly comment on them.
    \begin{itemize}
        \item \textit{Tangent Stokes}. For $\G$ of class $C^2$, the TS reformulation is well-posed in the energy setting; see Corollary~\ref{cor:well-posed} (well-posed reformulations). This result is unconditional because the well-posedness of the primal TS problem is known \cite[Theorem~4.3]{JankuhnOlshanskiiReusken2018}. Compared with \cite{NochettoShakipov2025}, we require $\G$ to be of class $C^2$ instead of $C^3$, and the viscous term is governed by $\DeltaS$ itself. The kernel of $\DeltaS$ consists of the Killing fields of $\G$, i.e., the intrinsic rigid rotations of $\G$. We handle these fields by working in their $L^2$-orthogonal complement.
 
        \item \textit{Transient problems}. For $\G$ of class $C^2$, the TTS reformulation is well-posed with stability constants independent of the final time $T$. For $d = 2$ and $\G$ of class $C^3$, the same framework covers the nonlinear TTNS problem in a strong setting; see Theorem~\ref{thm:reformulations} (equivalence of reformulations) and Corollary~\ref{cor:well-posed} (well-posed reformulations). The primal problem constrains its data by the compatibility condition $\divG \bf_0 = f(0)$. The reformulation needs no such condition; see Remark~\ref{rmk:original-ic} (original initial condition). For both transient problems, we rely on the well-posedness of the primal problem, which is established in the companion paper \cite{NochettoShakipov2026-reform-part-ii}.
 
        \item \textit{Equal-order formulations}. If the datum of the pressure Poisson equation lies in $H^{-1}$, the dual of $H^1$, then the pressure lies in $H^1$. The pressure Poisson equation may then be tested with $H^1$-functions, and the reformulated operator becomes a linear homeomorphism between $H^1$-based spaces; see Remark~\ref{rmk:equal-order} (equal-order formulations). Consequently, $H^1$-conforming finite elements suffice for both the velocity and the pressure, and conforming discretizations do not require a discrete inf-sup condition \cite{NochettoShakipov2025}.
 
        \item \textit{Interpretation}. The modified divergence datum in our main result solves a single linear Poisson or heat problem driven by the data; see Remark~\ref{rmk:factorization} (interpretation). Hence, perturbing the data of the pressure Poisson equation is equivalent to perturbing the divergence equation alone. At the discrete level, we therefore expect a consistency error in the pressure Poisson equation to appear as a defect in mass conservation. The size of this defect is controlled by the stability of the auxiliary Poisson or heat problem.
    \end{itemize}
 
    The abstract argument is simple, and so is the algebra behind the factorization $\cB = \cD \circ \cA$. The main subtlety lies in the choice of the functional setting. The spaces must be such that the primal problem can be expected to be well-posed, the data map is invertible, and the formal derivation of the pressure Poisson equation is rigorous. The right-hand side of the pressure Poisson equation \eqref{eq:intro-pressure-poi} foreshadows the difficulty: the operator $\alpha_1 \pa_t - 2 \mu \DeltaG$ acts on the divergence datum $f$. Hence, the divergence datum must lie in the solution space of a Poisson or heat problem; in the transient case, the divergence datum must possess a time derivative. For TTS in the energy setting, for instance, the time derivative $\pa_t f$ must lie in $L^2(0, T; H^{-2})$. Table~\ref{tab:settings} lists the functional settings of all three problems.
 
    We stress that the equivalence between $\cA x = y$ and $\cB x = \cD y$ requires nothing of $\cA$. The reformulation therefore inherits whatever existence, uniqueness, and stability properties the primal problem enjoys. This inheritance holds even when the primal problem is not well-posed on the whole data space, as we expect for the stationary tangent Navier--Stokes equations; see Section~\ref{sec:conclusion}. Conversely, since the equivalence is two-sided, a direct proof of the well-posedness of a reformulation would also yield that of the corresponding primal problem.
 
    Imposing orthogonality to the Killing fields is delicate at the discrete level \cite{BonitoDemlowLicht2020, OlshanskiiShakipov2026}. For the transient problems, this difficulty can be avoided because the time derivative controls the Killing component of the velocity. Hence, TTS and TTNS may be posed on all tangential fields, at the expense of $T$-dependent constants for TTS. Remark~\ref{rmk:killing} (Killing fields) explains why TTS may also be posed on the orthogonal complement without loss of generality, whereas TTNS may not.
 
    We now place our work in the literature. Replacing the incompressibility constraint with a pressure Poisson equation is a classical idea. On the torus or on the whole Euclidean space, the pressure Poisson equation recovers the pressure once the velocity is known \cite[\S 5]{robinson_three-dimensional_2016}. In bounded domains, however, the pressure Poisson equation must be supplemented with boundary conditions for the pressure, and the equivalence with the original problem hinges on their choice \cite[\S 5.4]{robinson_three-dimensional_2016}, \cite{GreshoSani1987}. The closest precedent to our work is the unconstrained formulation of the Navier--Stokes equations in bounded domains with no-slip boundary conditions \cite{GrubbSolonnikov1991, JohnstonLiu2004, LiuLiuPego2007}. There, the Neumann condition for the pressure is chosen so that the divergence of the velocity solves a heat problem with a homogeneous Neumann condition. Incompressibility is thus propagated from the initial data rather than enforced as a constraint. This mechanism leads to efficient schemes that treat the pressure explicitly and decouple the computation of the velocity and the pressure \cite{JohnstonLiu2004, ShirokoffRosales2011}.
 
    Our reformulations share the mechanism of the unconstrained formulation in the transient case: the divergence of the velocity solves a heat problem driven by the data. Otherwise, the two approaches have different aims. In \cite{LiuLiuPego2007}, the pressure is expressed in terms of the velocity, which yields an evolution equation for the velocity alone. Our reformulations keep the pressure as an unknown and replace the divergence equation with an elliptic equation for the pressure. The equivalence with the primal problem then rests on the factorization $\cB = \cD \circ \cA$ rather than on time evolution. Hence, the equivalence also covers the stationary problem and arbitrary data of the pressure Poisson equation.
 
    The pressure Poisson equation also underlies stabilized finite element methods for the Stokes problem \cite{BrezziPitkaranta1984, HughesFranca1986}. These methods add to the divergence equation a mesh-dependent multiple of the pressure Laplacian or of the momentum residual tested with pressure gradients. Unlike our reformulations, these methods perturb the saddle-point structure rather than remove it. For the Stokes problem with periodic boundary conditions, a nonstandard nonconforming approximation was proposed in \cite{Guermond2005}; the method of \cite{NochettoShakipov2025} can be viewed as a generalization thereof. On closed hypersurfaces, the pressure Poisson equation was used to prove higher regularity for two-dimensional TS \cite{OlshanskiiReuskenZhiliakov2021}. The same equation yields higher $L^p$-based regularity for TS and the stationary tangent Navier--Stokes equations in higher dimensions \cite{BenavidesNochettoShakipov2025}. To the best of our knowledge, however, the reformulations, viewed as problems in their own right, appear only in \cite{NochettoShakipov2025}, and only for TS.
 
    The rest of the paper is organized as follows. Section~\ref{sec:notation-and-function-spaces} introduces the geometric notation and the function spaces, and discusses the Killing fields of $\G$. Section~\ref{sec:preliminary-results} collects preliminary results: equivalent norms, Lions--Magenes-type results, the well-posedness of the heat problem on $\G$ at two levels of regularity, and the decomposition of $\DeltaS$ behind the pressure Poisson equation. Section~\ref{sec:reformulations} derives the reformulations formally, develops the abstract framework, and applies the framework to TS, TTS, and TTNS. Section~\ref{sec:conclusion} closes the paper with conclusions.

\section{Notation and function spaces} \label{sec:notation-and-function-spaces}
    In this section, we introduce the notation used throughout the paper. We denote by $T > 0$ the final time and by $I := (0, T)$ the corresponding time interval.
    \subsection{Geometry and differential operators} \label{ch:prelims;se:gen-notation;subse:geometry}
        Let $\G$ be a connected and closed (i.e., compact and without boundary) $d$-dimensional hypersurface of class $C^2$ embedded in $\bbR^{d+1}$. This is the baseline regularity assumption on $\G$ in this work; whenever more regularity is needed, we shall state it explicitly. Associated with $\G$, we let $\bn$ be its outward unit normal $\bn \in \bC^1(\G)$ and $\bP := \bI - \bn \otimes \bn$.

        Let $v, \bv, \bA$ be scalar, vector, and tensor-valued functions of class $C^2$ defined on $\G$. Let $\nablaG v$ denote the tangential gradient. For vector fields $\bv = (v_i)_{i=1}^{d+1}$, the covariant gradient is given by $\nablaG \bv = \bP (\nablaG^T v_i)_{i=1}^{d+1}$, and its symmetric part is denoted by $\DG \bv = \frac12 (\nablaG \bv + \nablaG^T \bv)$. The tangential divergence is defined as $\divG \bv := \tr (\nablaG \bv)$. For tensor fields, $\divG$ is defined row-wise: $\divG \bA := (\divG \bA_{i,\cdot})_{i=1}^{d+1}$. The Laplace--Beltrami operator reads $\DeltaG v := \divG \nablaG v$, the surface diffusion operator is $\DeltaS \bv := \bP \divG \DG \bv$, while the Bochner Laplacian is $\Delta_B \bv := \bP \divG \nablaG \bv$.

        We next introduce further geometric quantities. Let $\bW := \nablaG \bn$ be the Weingarten map, which is symmetric. The following quantity, called the Ricci curvature map, will play a crucial role in the forthcoming development:
        \[
            \bR := \tr(\bW) \bW - \bW \bW.
        \]
        For $d = 2$, we have $\bR = K \bP$ \cite[Lemma A.3]{JankuhnOlshanskiiReusken2018}, where $K$ is the Gauss curvature.

    \subsection{Function spaces} \label{ch:prelims;se:general-notation;subse:function-spaces}
        Let $\G$ be of class $C^{\max\{m, 2\}}, m \in \bbZ_{\geq 1}$. By $H^m(\G)$, $m \in \bbZ_{\geq 1}$, we denote the Sobolev spaces on $\G$, and by $\bH^m(\G)$ and $\dH^m(\G)$ their vector- and tensor-valued counterparts, respectively. We shall use this notation consistently for other function spaces as well. We additionally write $H^0(\G) := L^2(\G)$ for the Lebesgue space on $\G$. We refer to \cite{BenavidesNochettoShakipov2026} for a careful discussion of the various definitions of Sobolev spaces on manifolds of limited regularity.

        We shall use the parentheses notation $(\cdot, \cdot)$ for the $L^2(\G)$-, $\bL^2(\G)$-, and $\dL^2(\G)$-inner products, and $\|\cdot\|$ for the respective norms. We next introduce some important closed subspaces for $m \in \bbZ_{\geq 1}$:
        \begin{itemize}
            \item \textbf{Scalar fields with vanishing mean}. We define $L^2_\# := \{v \in L^2(\G): (v, 1) = 0\} \subseteq L^2(\G)$, $H^m_\# := L^2_\# \cap H^m(\G)$, and write $H^{-m}_\# = (H^m_\#)'$ for the dual of $H^m_\#$. We similarly define $C^m_\#$ for $m \geq 0$. Moreover, the following Poincar\'e inequality is valid for Lipschitz $\G$ \cite{DziukElliott2013, BonitoDemlowNochetto2020}:
            \begin{equation} \label{eq:poincare}
                \|v\| \leq C_P \|\nablaG v\|, \quad \forall v \in H^1_\#.
            \end{equation}
            \item \textbf{Tangential vector fields}. We define $\bL^2_t := \{\bv \in \bL^2(\G) : \bv \cdot \bn = 0 \text{ on $\G$}\} \subseteq \bL^2(\G)$, and $\bH^m_t := \bL^2_t \cap \bH^m(\G)$, and write $\bH^{-m}_t := (\bH^m_t)'$. We similarly define $\bC^m_t$ for $m \geq 0$.
        \end{itemize}

    \subsection{Killing vector fields} \label{subsec:killing}
        In flat domains, the Korn inequality is classical \cite{Horgan1995, AcostaDuran2017}. For $\G$ of class $C^2$, the following Korn inequality is a consequence of \cite[Lemma 4.1]{JankuhnOlshanskiiReusken2018}, which is stated for $d = 2$ only, even though its proof is valid for $d \geq 2$:
        \begin{equation} \label{eq:korn-type}
            \|\bv\|_{\bH^1(\G)} \leq C_K (\|\bv\| + \|\DG \bv\|), \quad \forall \bv \in \bH^1_t.
        \end{equation}
        Let $\cK := \{\bv \in \bH^1_t : \DG \bv = \bzero\}$ be the space of Killing vector fields, i.e., the intrinsic rigid rotations of $\G$. From \eqref{eq:korn-type}, we find that $\|\bv_\cK\|_{\bH^1(\G)} \leq C_K \|\bv_\cK\|$ for all $\bv_\cK \in \cK$, so the unit ball in $\cK$ is compact by Rellich--Kondrashov, whence $\cK$ is finite-dimensional by the Riesz theorem. Since $\cK \subseteq \bH^1_t \subseteq \bL^2_t$ is finite-dimensional, we can write the orthogonal decomposition
        \begin{equation} \label{eq:decomposition}
            \bL^2_t = \cK \oplus \cK^\perp.
        \end{equation}
        For the same reason all norms on $\cK$ are equivalent, and decomposition \eqref{eq:decomposition} is stable in $\bH^1(\G)$ for $\G$ of class $C^2$.

        We next let $\cP_\cK : \bL^2_t \to \cK$ be the $\bL^2_t(\G)$-orthogonal projection onto $\cK$, and set $\cP_\cK^\perp := \operatorname{id} - \cP_\cK$. This allows us to define another class of subspaces: for $k \in \bbZ_{\geq 1}$,
        \begin{equation*}
                \bL^2_{t, \perp} := \{\bv \in \bL^2_t : \cP_\cK \bv = \bzero\}, \quad \bH^k_{t, \perp} := \{\bv \in \bH^k_t : \cP_\cK \bv = \bzero\}, \quad \bH^{-k}_{t, \perp} := (\bH^k_{t, \perp})',
        \end{equation*}
        and we use the convention $\bH^0_{t, \perp} := \bL^2_{t, \perp}$. A standard contradiction argument combined with \eqref{eq:korn-type} then yields the following Korn inequality:
        \begin{equation} \label{eq:korn}
            \|\bv\|_{\bH^1(\G)} \leq C_K' \|\DG \bv\|, \quad \forall \bv \in \bH^1_{t, \perp}.
        \end{equation}
        Since $\divG = \tr(\DG)$, we have $\divG \bv_\cK = 0$ for all $\bv_\cK \in \cK$. This leads to the following crucial observation: $\divG \bv_\cK = 0$ implies that $(\nablaG \phi, \bv_\cK) = 0$ for all $\bv_\cK \in \cK$, whence, for $\ell \in \{0,1\}$,
        \[
            \nablaG \phi \in \bH^\ell_{t, \perp}, \quad \forall \phi \in H^{\ell+1}_\#.
        \]
        We next show that \eqref{eq:decomposition} is also stable in $\bH^k(\G)$ for $k \geq 2$ and $\G$ of class $C^{k+1}$. To this end, we recall the identity $\bP \divG \nablaG^T \bv = \nablaG \divG \bv + \bR \bv$ for all $\bv \in \bC^2_t$, which is valid for $\G$ of class $C^3$ \cite[Lemma 2.1]{JankuhnOlshanskiiReusken2018}. A density argument as in \cite[Appendix B, proof of Lemma 2.13]{BenavidesNochettoShakipov2025}, followed by the density of $\bC^1_t$ in $\bH^1_t$, extends it to $\G$ of class $C^2$ in the following weak form:
        \begin{equation} \label{eq:identity}
            (\nablaG^T \bv, \nablaG \bw) = (\divG \bv, \divG \bw) - (\bR \bv, \bw)
        \end{equation}
        for all $\bv, \bw \in \bH^1_t$. Choosing $\bv = \bv_\cK \in \cK$ then leads to $(\nablaG \bv_\cK, \nablaG \bw) = (\bR \bv_\cK, \bw)$ for all $\bw \in \bH^1_t$, which is a weak formulation of the Bochner--Laplace problem $- \Delta_B \bv_\cK = \bR \bv_\cK$. By the $\bH^2(\G)$-regularity of the Bochner--Laplace problem, valid for $\G$ of class $C^3$ \cite{BenavidesNochettoShakipov2025}, we deduce that $\bv_\cK \in \bH^2(\G)$; bootstrapping then yields higher $\bH^k(\G)$-regularity for $k \geq 3$ and $\G$ of class $C^{k+1}$.

        As a consequence, the density of $\bH^\ell_t$ in $\bH^{\ell-1}_t$, $\ell \in \{1,2\}$, valid for $\G$ of class $C^{\ell+1}$ \cite{BenavidesNochettoShakipov2025}, together with the $\bH^\ell(\G)$-stability of \eqref{eq:decomposition} implies that $\bH^\ell_{t,\perp}$ is dense in $\bH^{\ell-1}_{t, \perp}$.

    \subsection{Norms and duality pairings} \label{ch:prelims;se:general-notation;subse:inner-prod}
        \textbf{Norms}. For a generic Banach space $\sX$, we denote its norm by $\|\cdot\|_\sX$. We endow $L^2_\# \subseteq L^2(\G)$ and $\bL^2_t \subseteq \bL^2(\G)$ with the norms of their respective superspaces.
        Let $k \in \bbZ_{\neq 0}$ and $\sV_k \in \{H^k_\#, \bH^k_t, \bH^k_{t, \perp}\}$. Whenever there is no ambiguity, we shall use the following shorthand notation for norms on $\sV_k$:
        \begin{equation}
            \|v\|_{k} := \|v\|_{\sV_k}, \qquad \forall v \in \sV_k.
        \end{equation}
        Here, $\|\cdot\|_{\sV_k}$ is a convenient equivalent norm on $\sV_k$, which we shall specify in Section~\ref{subsec:equiv-norms}.

        \textbf{Duality pairings}. For a generic Banach space $\sX$, we denote by $\langle \cdot, \cdot \rangle_{\sX' \times \sX}$ the duality pairing between $\sX$ and its dual $\sX'$. Let $k \in \bbZ_{\neq 0}$ and $\sV_k \in \{H^k_\#, \bH^k_t, \bH^k_{t, \perp}\}$. Whenever there is no ambiguity, we shall use the following shorthand notation for the duality pairings between $\sV_k$ and $(\sV_k)' \cong \sV_{-k}$ (valid also for $k < 0$ due to reflexivity): for all $u \in \sV_{-k}, v \in \sV_k$,
        \begin{equation} \label{eq:duality-pairing}
            \langle u, v \rangle_{-k, k} := \langle u, v \rangle_{\sV_{-k} \times \sV_k}.
        \end{equation}
        It will sometimes be convenient to write these duality pairings in the reverse order as well: $\langle v, u \rangle_{k, -k} := \langle u, v \rangle_{-k, k}$.

    \subsection{Bochner--Sobolev spaces} \label{subsec:boch-sobolev-spaces}
        For transient problems, we shall use the following Bochner--Sobolev spaces for $r=-2,-1,0,1,2$ and $s=0,1,2$:
        \begin{equation} \label{eq:spaces-cVr-cWr}
            \begin{split}
                \cV^r_\# := L^2(I; H^r_\#), & \quad \cW^s_\# := \cV^s_\# \cap H^1(I; H^{s-2}_\#), \\
                \bcV^r_t := L^2(I; \bH^r_t), & \quad \bcW^s_t := \bcV^s_t \cap H^1(I; \bH^{s-2}_t), \\
                \bcV^r_{t, \perp} := L^2(I; \bH^r_{t, \perp}), & \quad \bcW^s_{t, \perp} := \bcV^s_{t, \perp} \cap H^1(I; \bH^{s-2}_{t, \perp}).
            \end{split}
        \end{equation}
        Here, the $\cW, \bcW$-spaces are the natural choice for solutions of linear parabolic problems. By convention, the $\cW_\#^s$- and $\bcW_{t, \perp}^s$-norms include the trace at $t = 0$:
        \begin{equation} \label{eq:cW-norms}
            \begin{split}
                \|v\|_{\cW^s_\#} & := \left( \|v(0)\|_{s-1}^2 + \|\pa_t v\|_{\cV^{s-2}_\#}^2 + \|v\|_{\cV^s_\#}^2 \right)^{1/2}, \\
                \|\bv\|_{\bcW^s_{t, \perp}} & := \left( \|\bv(0)\|_{s-1}^2 + \|\pa_t \bv\|_{\bcV^{s-2}_{t,\perp}}^2 + \|\bv\|_{\bcV^s_{t,\perp}}^2 \right)^{1/2};
            \end{split}
        \end{equation}
        this is justified by the Lions--Magenes-type results of Section~\ref{subsec:lions-magenes}, with embedding constants that are uniform in $T$. We finally note that, for $r = 0,1,2$,
        \begin{equation}
            [\cV^r_\#]' \cong \cV^{-r}_\#, \qquad [\bcV^r_{t, \perp}]' \cong \bcV^{-r}_{t, \perp}.
        \end{equation}
        This is not the case for $\cW^s_\#$ and $\bcW^s_{t, \perp}$, whose duals have a more complicated structure.

\section{Preliminary results} \label{sec:preliminary-results}
    In this section, we collect the analytical tools needed in Section~\ref{sec:reformulations}. We first introduce spectral norms, in which $\Lambda^{1/2} = (-\DeltaG)^{1/2}$ maps each space of the Sobolev scale isometrically onto the next one. This observation reduces the Lions--Magenes lemma and the heat equation on $\G$ at shifted levels of regularity to their classical versions. We close this section by recalling a crucial geometric identity, which is the main engine behind the reformulations.
    \subsection{Equivalent norms} \label{subsec:equiv-norms}
        We now define convenient equivalent norms on $H^k_\#$ and $\bH^k_{t, \perp}$, while the spaces $H^1(\G)$ and $\bH^1(\G)$ are endowed with their full norms. Following \cite[Appendix~A.2]{OlshanskiiShakipov2026}, we define these norms spectrally. Throughout the paper, we write $\Lambda := -\DeltaG$ and $\Lambda_S := -\DeltaS$, and we use the conventions $H^0_\# := L^2_\#$ and $\langle \cdot, \cdot \rangle_{0,0} := (\cdot, \cdot)$.
 
        \textit{Scalar fields}. Since $\G$ is connected and closed, the Poincar\'e inequality \eqref{eq:poincare} and the compactness of the embedding $H^1_\# \embeds L^2_\#$ yield an $L^2$-orthonormal basis $\{\psi_i\}_{i \geq 1}$ of $L^2_\#$ consisting of eigenfunctions of $\Lambda$, with eigenvalues $0 < \lambda_1 \leq \lambda_2 \leq \dots \to \infty$ \cite[Theorem~6.11]{Brezis2011}. By the $H^2$-regularity of the Laplace--Beltrami problem \cite[Theorem~1.2]{BenavidesNochettoShakipov2026}, these eigenfunctions belong to $H^2_\#$. Let $s \in \{-2, -1, 0, 1, 2\}$ and $v \in H^s_\#$. We denote by $v_i := \langle v, \psi_i \rangle$ the coefficients of $v$, and we let $\Lambda^{s/2}$ multiply the $i$-th coefficient by $\lambda_i^{s/2}$. We then set
        \begin{equation} \label{eq:equiv-norms}
            \|v\|_s := \|\Lambda^{s/2} v\| = \left( \textstyle\sum_{i \geq 1} \lambda_i^s v_i^2 \right)^{1/2}, \qquad (u, w)_{H^s_\#} := (\Lambda^{s/2} u, \Lambda^{s/2} w).
        \end{equation}
        By the eigenvalue relation, $\|v\|_1 = \|\nablaG v\|$ and $\|v\|_2 = \|\DeltaG v\|$, and $\|\cdot\|_0$ is the $L^2$-norm. By the Cauchy--Schwarz inequality, the negative norms are the dual norms:
        \begin{equation} \label{eq:equiv-norms-dual}
            \|f\|_{-k} = \sup_{0 \neq v \in H^k_\#} \frac{\langle f, v \rangle_{-k,k}}{\|v\|_k}, \qquad k \in \{1, 2\}.
        \end{equation}
        Hence, by the Poincar\'e inequality and the $H^2$-regularity, the norm $\|\cdot\|_s$ is an equivalent norm on $H^s_\#$ for each $s \in \{-2, -1, 0, 1, 2\}$.
 
        With these norms, $\|\Lambda^{1/2} v\|_{s-1} = \|v\|_s$ by definition. In other words, $\Lambda^{1/2}$ maps each space of the scale isometrically onto the next one:
        \begin{equation} \label{eq:sqrt-isometry}
            \Lambda^{1/2} : H^s_\# \to H^{s-1}_\# \ \text{is an isometric isomorphism}, \qquad s \in \{-1, 0, 1, 2\}.
        \end{equation}
        Applying the isometry \eqref{eq:sqrt-isometry} twice shows that $\Lambda$ maps $H^s_\#$ isometrically onto $H^{s-2}_\#$ for $s \in \{0, 1, 2\}$. Here, $\Lambda$ is understood strongly for $s = 2$, weakly via $\langle \Lambda v, w \rangle_{-1,1} := (\nablaG v, \nablaG w)$ for $s = 1$, and ultraweakly via $\langle \Lambda v, w \rangle_{-2,2} := (v, \Lambda w)$ for $s = 0$. Each of these definitions extends the previous one by integration by parts and multiplies the $i$-th coefficient by $\lambda_i$. The operator $\Lambda^{1/2}$ does not depend on time. Hence, $\Lambda^{1/2}$ acts pointwise in time on the Bochner--Sobolev spaces of Section~\ref{subsec:boch-sobolev-spaces}, maps $\cV^r_\#$ isometrically onto $\cV^{r-1}_\#$ for $r \in \{-1, 0, 1, 2\}$, and commutes with $\pa_t$. Since $\Lambda^{1/2}$ also commutes with $\Lambda$, these properties reduce the shifted regularity setting $r \in \{0, 2\}$ of Sections~\ref{subsec:lions-magenes} and~\ref{subsec:parab-theory} to the classical level $r = 1$.
 
        \textit{Tangential fields}. Korn's inequality \eqref{eq:korn} and the compactness of the embedding $\bH^1_{t, \perp} \embeds \bL^2_{t, \perp}$ likewise yield an orthonormal basis of $\bL^2_{t, \perp}$ consisting of eigenfunctions of $\Lambda_S$ with positive eigenvalues. We define the norms $\|\bv\|_s$ and the inner products $(\bu, \bw)_{\bH^s_{t, \perp}}$ by the formulas \eqref{eq:equiv-norms}, with $\Lambda$ replaced by $\Lambda_S$. Then $\|\bv\|_1 = \|\DG \bv\|$ and $\|\bv\|_2 = \|\DeltaS \bv\|$, and the negative norms are the dual norms. For $\G$ of class $C^2$, Korn's inequality shows that $\|\cdot\|_s$ is an equivalent norm on $\bH^s_{t, \perp}$ for $s \in \{-1, 0, 1\}$. For $\G$ of class $C^3$, the $\bH^2$-regularity of $\DeltaS$ \cite[Theorem 6.1]{BenavidesNochettoShakipov2025} extends this statement to $s \in \{-2, 2\}$. Here, the operator $\DeltaS$ maps $\bH^2_{t, \perp}$ into $\bL^2_{t, \perp}$, because its range is orthogonal to the Killing fields: $(-\DeltaS \bv, \bw_\cK) = (\DG \bv, \DG \bw_\cK) = 0$ for all $\bw_\cK \in \cK$. Whenever all spaces involved carry equivalent norms, the isometry \eqref{eq:sqrt-isometry} and its consequences hold verbatim for $\Lambda_S^{1/2}$ and $\Lambda_S$. For $k \in \mathbb{Z}_{\geq 3}$, we shall use the natural (full) norms, and the negative norms defined accordingly.
 
    \subsection{Lions--Magenes-type results} \label{subsec:lions-magenes}
        For $r = 1$, the following result is the classical Lions--Magenes lemma in the Gelfand triple $H^1_\# \embeds L^2_\# \embeds H^{-1}_\#$. The isometry \eqref{eq:sqrt-isometry} extends the lemma to the levels $r \in \{0, 2\}$.
        \begin{proposition}[Lions--Magenes] \label{prop:lm-cW}
            Let $\G$ be of class $C^2$, let $r \in \{0,1,2\}$, and let $v \in \cW^r_\#$. Then (after possibly redefining on a set of measure zero) $v \in C^0([0,T]; H^{r-1}_\#)$, and the following stability bound is valid:
            \begin{equation} \label{eq:lm-stab}
                \max_{0 \leq t \leq T} \|v(t)\|_{r-1}^2 \leq \|v(0)\|_{r-1}^2 + \|\pa_t v\|_{\cV^{r-2}_\#}^2 + \|v\|_{\cV^r_\#}^2.
            \end{equation}
        \end{proposition}
        \begin{proof}
            For $r = 1$, the classical Lions--Magenes lemma \cite[Chapter~III, Lemma~1.2]{Temam1979} yields $v \in C^0([0,T]; L^2_\#)$ and $\frac{\d}{\dt} \|v\|^2 = 2 \langle \pa_t v, v \rangle_{-1,1}$ a.e.\ in $I$. Integrating this identity in time and using the Cauchy--Schwarz and Young inequalities yields the bound \eqref{eq:lm-stab}. For $r \in \{0, 2\}$, let $w := \Lambda^{(r-1)/2} v$ for $v \in \cW^r_\#$. By the isometry \eqref{eq:sqrt-isometry}, applied pointwise in time, we have $w \in \cW^1_\#$ with $\|\pa_t w\|_{\cV^{-1}_\#} = \|\pa_t v\|_{\cV^{r-2}_\#}$ and $\|w\|_{\cV^1_\#} = \|v\|_{\cV^r_\#}$. The case $r = 1$ thus yields $w \in C^0([0,T]; L^2_\#)$ and the bound \eqref{eq:lm-stab} for $w$. Since $\Lambda^{(1-r)/2}$ maps $L^2_\#$ isometrically onto $H^{r-1}_\#$, both assertions carry over to $v = \Lambda^{(1-r)/2} w$.
        \end{proof}
        The bound \eqref{eq:lm-stab} justifies the convention \eqref{eq:cW-norms} of Section~\ref{subsec:boch-sobolev-spaces}, whereby the $\cW$-norms include the trace at $t = 0$, with embedding constants that are uniform in $T$. Moreover, $\Lambda^{1/2}$ commutes with the evaluation at $t = 0$, and hence maps $\cW^r_\#$ isometrically onto $\cW^{r-1}_\#$ for $r \in \{1, 2\}$, with the norms \eqref{eq:cW-norms}.
        \begin{remark}[vector-valued counterparts of Lions--Magenes results] \label{rmk:lm-vector}
            Proposition~\ref{prop:lm-cW} (Lions--Magenes) and its proof remain valid verbatim for $\Lambda_S$ on the spaces $\bH^r_{t, \perp}, \bcV^r_{t, \perp}, \bcW^r_{t, \perp}$, with the norms of Section~\ref{subsec:equiv-norms}. These counterparts hold for $r = 1$ if $\G$ is of class $C^2$, and for $r \in \{0, 1, 2\}$ if $\G$ is of class $C^3$. The same is true for the Bochner--Laplace operator $\Lambda_B := - \Delta_B = - \bP \divG \nablaG$ on the spaces $\bH^r_t, \bcV^r_t, \bcW^r_t$, with norms defined as in the spectral norms \eqref{eq:equiv-norms}. Here, the norm $\|\nablaG \cdot\|$ is equivalent on $\bH^1_t$ for $\G$ of class $C^2$ by the Poincar\'e inequality $\|\bv\| \leq C_{P,t} \|\nablaG \bv\|$ on $\bH^1_t$ \cite{HansboLarsonLarsson2020, BenavidesNochettoShakipov2025}. The norm $\|\Lambda_B \cdot\|$ is equivalent on $\bH^2_t$ for $\G$ of class $C^3$ \cite[\S3]{BenavidesNochettoShakipov2025}.
 
            In particular, for $\G$ of class $C^2$, functions in $\bcW^1_{t, \perp}$ and $\bcW^1_t$ are continuous in time with values in $\bL^2_{t, \perp}$ and $\bL^2_t$, respectively. For $\G$ of class $C^3$, functions in $\bcW^2_{t, \perp}$ and $\bcW^2_t$ are continuous in time with values in $\bH^1_{t, \perp}$ and $\bH^1_t$, respectively.
        \end{remark}
 
    \subsection{Parabolic theory} \label{subsec:parab-theory}
        We now consider the heat equation $\pa_t \varphi + 2 \mu \Lambda \varphi = h$ with initial condition $\varphi(0) = h_0$, at two levels of regularity: the weak one ($r = 1$) and the ultraweak one ($r = 0$). The latter is instrumental in the reformulation of the transient tangent Stokes equations; see the proof of Lemma~\ref{lem:data-maps} (data maps). The weak level is classical, and the ultraweak level follows from it by the isometry \eqref{eq:sqrt-isometry}.
        \begin{proposition}[heat equation] \label{prop:heat}
            Let $\G$ be of class $C^2$ and $r \in \{0,1\}$. Given data $(h_0, h) \in H^{r-1}_\# \times \cV^{r-2}_\#$, there exists a unique $\varphi \in \cW^r_\#$ such that
            \begin{equation} \label{eq:heat}
                \varphi(0) = h_0 \ \text{ in } H^{r-1}_\#, \qquad \pa_t \varphi + 2 \mu \Lambda \varphi = h \ \text{ in } \cV^{r-2}_\#,
            \end{equation}
            where $\Lambda$ is understood weakly for $r = 1$ and ultraweakly for $r = 0$. Furthermore, there exist $C_1, C_2 > 0$ depending only on $\mu$ (and not on $T$) such that
            \begin{equation} \label{eq:heat-bounds}
                C_1 \|(h_0, h)\|_{H^{r-1}_\# \times \cV^{r-2}_\#} \leq \|\varphi\|_{\cW^r_\#} \leq C_2 \|(h_0, h)\|_{H^{r-1}_\# \times \cV^{r-2}_\#},
            \end{equation}
            where $\|\varphi\|_{\cW^r_\#} = \left( \|\varphi(0)\|_{r-1}^2 + \int_I \|\pa_t \varphi\|_{r-2}^2 + \|\varphi\|_r^2 \dt \right)^{1/2}$ is the $\cW$-norm \eqref{eq:cW-norms}.
        \end{proposition}
        \begin{proof}
            For $r = 1$, existence and uniqueness follow from the Lions theorem \cite[Vol.~I, Ch.~3, \S4]{LionsMagenes1972}, and the two-sided estimate \eqref{eq:heat-bounds} follows from the standard energy estimate. The constants depend only on $\mu$, because the bilinear form $2 \mu (\nablaG \cdot, \nablaG \cdot)$ is coercive on $H^1_\#$ endowed with $\|\cdot\|_1 = \|\nablaG \cdot\|$. For $r = 0$, the operator $\Lambda^{1/2}$ commutes with $\Lambda$. By the isometry \eqref{eq:sqrt-isometry} and Proposition~\ref{prop:lm-cW} (Lions--Magenes), the operator $\Lambda^{1/2}$ maps the solution space $\cW^1_\#$ and the data space $L^2_\# \times \cV^{-1}_\#$ of the case $r = 1$ isometrically onto those of the case $r = 0$. Hence, $\varphi$ solves the heat equation \eqref{eq:heat} with $r = 0$ and data $(h_0, h)$ if and only if $\Lambda^{-1/2} \varphi$ solves it with $r = 1$ and data $(\Lambda^{-1/2} h_0, \Lambda^{-1/2} h)$, and all norms in the two-sided estimate \eqref{eq:heat-bounds} are preserved.
        \end{proof}

    \subsection{Decomposition of the vector Laplacian} \label{subsec:stokes-iden}
        The following lemma, proven in \cite[Lemma 2.13]{BenavidesNochettoShakipov2025}, states the commutator identity between the tangential divergence and $\DeltaS$. This identity produces the cancellation behind the reformulations; see the commutator identity \eqref{eq:commutator}. Given its importance, below we present its short proof for completeness.
        \begin{lemma}[decomposition of $\DeltaS$] \label{lem:stokes-iden}
            Let $\G$ be of class $C^2$. Then for all $(\bv, \phi) \in \bH^1_t \times H^2_\#$,
            \begin{equation} \label{eq:stokes-iden-weak}
                (\DG \bv, \DG (\nablaG \phi)) = (\divG \bv, \DeltaG \phi) - (\bR \bv, \nablaG \phi),
            \end{equation}
            where $\bR = \tr(\bW)\bW - \bW \bW$ with $\bW$ being the Weingarten map. If $\G$ is of class $C^3$, then the following stronger version of \eqref{eq:stokes-iden-weak} is valid for all $(\bv, \phi) \in \bH^2_t \times H^1_\#$:
            \begin{equation} \label{eq:stokes-iden-stronger}
                (\DeltaS \bv, \nablaG \phi) = (\nablaG \divG \bv + \bR \bv, \nablaG \phi).
            \end{equation}
            Finally, if $\G$ is of class $C^4$, then the strong version of \eqref{eq:stokes-iden-stronger} is valid a.e.\ on $\G$ for all $\bv \in \bH^3_t$:
            \begin{equation} \label{eq:stokes-iden-strong}
                \divG (\DeltaS \bv) = \DeltaG (\divG \bv) + \divG(\bR \bv).
            \end{equation}
        \end{lemma}
        \begin{proof}
            Recall identity \eqref{eq:identity}: for all $\bv, \bw \in \bH^1_t$, $(\nablaG^T \bv, \nablaG \bw) = (\divG \bv, \divG \bw) - (\bR \bv, \bw)$. Picking $\bw = \nablaG \phi$ for $\phi \in H^2_\#$ yields
            \begin{equation*}
                (\nablaG^T \bv, \nablaG (\nablaG \phi)) = (\divG \bv, \DeltaG \phi) - (\bR \bv, \nablaG \phi).
            \end{equation*}
            Since the covariant Hessian $\nablaG (\nablaG \phi)$ is symmetric \cite[Proposition B.2]{BenavidesNochettoShakipov2025}, we deduce that $\nablaG (\nablaG \phi) = \DG(\nablaG \phi)$ and consequently infer \eqref{eq:stokes-iden-weak}. Identities \eqref{eq:stokes-iden-stronger} and \eqref{eq:stokes-iden-strong} then follow by integration by parts.
        \end{proof}

\section{Reformulation} \label{sec:reformulations}
    In this section, we derive and analyze reformulations of the tangent Stokes (TS), transient tangent Stokes (TTS), and transient tangent Navier--Stokes (TTNS) equations. The reformulation arises by applying a linear and invertible \textit{data map} $\cD$ to both sides of the primal equation $\cA x = y$, much like a row operation in Gaussian elimination. The row operation is chosen so that the composition $\cD \circ \cA$ simplifies to the operator $\cB$ of the reformulation; the simplification rests on the commutator identity \eqref{eq:commutator} below. In $\cB$, the pressure Poisson equation replaces the divergence equation of $\cA$. Unlike the divergence equation, the pressure Poisson equation contains the pressure, through the invertible operator $-\DeltaG$. In this sense, the reformulation circumvents the saddle-point structure. Once the factorization $\cB = \cD \circ \cA$ is available, the equivalence of the primal problem and its reformulation is elementary: it follows by applying $\cD^{-1}$. The main delicate point is the functional setting. The spaces must be chosen so that the primal problem is well-posed, the data map is a homeomorphism, and each step of the formal computation is meaningful.

    We proceed as follows. In Section~\ref{subsec:reform-of-stokes-type-pbs}, we derive the reformulations formally and explain which functional settings to expect. Section~\ref{subsec:abstract-reformulation} records the abstract framework. In Section~\ref{subsec:settings}, we fix the functional settings for TS, TTS, and TTNS. We also show there that the data maps are homeomorphisms. Section~\ref{subsec:factorization} verifies $\cB = \cD \circ \cA$ in all three settings at once. Finally, Section~\ref{subsec:main-results} states the main results in an unconditional form: each reformulation is well-posed if and only if the corresponding primal problem is. It then comments on what is known about the primal problems.

    \subsection{Formal derivation} \label{subsec:reform-of-stokes-type-pbs}
        Recall the general fluid problem \eqref{eq:tts-ttns} from the introduction:
        \begin{subequations} \label{eq:tts-ttns-copy}
            \begin{align}
                \alpha_1 \pa_t \bu + \alpha_2 \bu \cdot \nablaG \bu - 2 \mu \DeltaS \bu + \nablaG u & = \bf, \label{eq:tts-ttns-mom} \\
                \divG \bu & = f, \label{eq:tts-ttns-div} \\
                \alpha_1 \bu(0) & = \alpha_1 \bf_0.
            \end{align}
        \end{subequations}
        In the transient case $\alpha_1 = 1$, the initial condition $\bu(0) = \bf_0$ and the divergence equation at $t = 0$ force the compatibility condition $\divG \bf_0 = f(0)$. This makes the data space a constrained subspace of a product space. The dual of such a space, which contains the test functions, is awkward to describe. We therefore relax the initial condition $\bu(0) = \bf_0$ using a Lagrange multiplier $u_0$ with vanishing mean:
        \begin{equation} \label{eq:new-ic}
            \bu(0) + \nablaG u_0 = \bf_0.
        \end{equation}
        Taking $-\divG$ of the relaxed initial condition \eqref{eq:new-ic} and using $\divG \bu(0) = f(0)$ yields
        \begin{equation} \label{eq:new-ic-pressure}
            - \DeltaG u_0 = f(0) - \divG \bf_0,
        \end{equation}
        so that $u_0 = 0$ if and only if the compatibility condition $\divG \bf_0 = f(0)$ holds. Thus, \eqref{eq:new-ic} is consistent with \eqref{eq:tts-ttns-copy}, but lets \textit{the data range over a full product space}. The augmented problem reads: given $y = (\bf_0, \bf, f)$, find $x = (u_0, \bu, u)$ such that
        \begin{equation} \label{eq:stokes-type-aug}
            \begin{cases}
                \bu(0) + \nablaG u_0 \kern-0.72em & = \bf_0, \\
                \alpha_1 \pa_t \bu + \alpha_2 \bu \cdot \nablaG \bu - 2 \mu \DeltaS \bu + \nablaG u \kern-0.72em & = \bf, \\
                \divG \bu \kern-0.72em & = f,
            \end{cases}
        \end{equation}
        which we write compactly as $\cA x = y$. If $\alpha_1 = 0$, then the problem is stationary and has no initial condition. In this case, the unknown $u_0$, the datum $\bf_0$, and the first equation in \eqref{eq:stokes-type-aug} are absent. The same convention applies to all formulas in this section: everything associated with the initial condition is dropped when $\alpha_1 = 0$.

        We now derive the pressure Poisson equation. We take $-\divG$ of the momentum equation in the augmented problem \eqref{eq:stokes-type-aug}. We then use $\divG \pa_t \bu = \pa_t \divG \bu$ and identity \eqref{eq:stokes-iden-strong}, i.e., $\divG (\DeltaS \bu) = \DeltaG (\divG \bu) + \divG(\bR \bu)$. These two identities yield
        \begin{equation} \label{eq:div-of-momentum}
            - (\alpha_1 \pa_t - 2 \mu \DeltaG) \divG \bu - \alpha_2 \divG (\bu \cdot \nablaG \bu) + 2 \mu \, \divG (\bR \bu) - \DeltaG u = - \divG \bf.
        \end{equation}
        Since $\divG \bu = f$, the first term equals $- (\alpha_1 \pa_t - 2 \mu \DeltaG) f$. Moving it to the right-hand side yields the pressure Poisson equation from \eqref{eq:intro-pressure-poi}:
        \begin{equation} \label{eq:intro-pressure-poi-copy}
            - \DeltaG u - \alpha_2 \divG (\bu \cdot \nablaG \bu) + 2 \mu \, \divG (\bR \bu) = - \divG \bf + (\alpha_1 \pa_t - 2 \mu \DeltaG) f.
        \end{equation}
        In other words, the pressure Poisson equation is the sum of two equations: the momentum equation with $-\divG$ applied to both sides, and the divergence equation with $\alpha_1 \pa_t - 2 \mu \DeltaG$ applied to both sides. Likewise, the initial Poisson equation \eqref{eq:new-ic-pressure} is the sum of the relaxed initial condition \eqref{eq:new-ic} with $-\divG$ applied to both sides and the divergence equation evaluated at $t = 0$. These are linear operations on the equations of the augmented problem \eqref{eq:stokes-type-aug}, performed on both sides at once, much like row operations in Gaussian elimination.

        The reformulated system keeps the initial condition \eqref{eq:new-ic} and the momentum equation \eqref{eq:tts-ttns-mom}. However, it \textit{replaces the divergence equation \eqref{eq:tts-ttns-div} with the pressure Poisson equation \eqref{eq:intro-pressure-poi-copy}}. If $\alpha_1 = 1$, it also adds the initial Poisson equation \eqref{eq:new-ic-pressure} at $t = 0$:
        \begin{equation} \label{eq:reformulations-eq}
            \begin{cases}
                \bu(0) + \nablaG u_0 \kern-0.72em & = \bf_0, \\
                \alpha_1 \pa_t \bu + \alpha_2 \bu \cdot \nablaG \bu - 2 \mu \DeltaS \bu + \nablaG u \kern-0.72em & = \bf, \\
                - \DeltaG u_0 \kern-0.72em & = f(0) - \divG \bf_0, \\
                - \DeltaG u - \alpha_2 \divG (\bu \cdot \nablaG \bu) + 2 \mu \, \divG (\bR \bu) \kern-0.72em & = - \divG \bf + \alpha_1 \pa_t f - 2 \mu \DeltaG f.
            \end{cases}
        \end{equation}
        The reformulated system \eqref{eq:reformulations-eq} results from applying the above operations to both sides of the augmented problem \eqref{eq:stokes-type-aug}. Hence, its right-hand side is $\cD y$. Here, the linear \textit{data map} $\cD$ records these operations; it is \textit{block triangular}:
        \begin{equation} \label{eq:data-map-formal-full}
            \cD
            \left[\begin{array}{c}
                \bf_0 \\
                \bf \\
                \hline
                f
            \end{array}\right]
            =
            \left[\begin{array}{cc|c}
                \operatorname{id} & 0 & 0 \\
                0 & \operatorname{id} & 0 \\
                \hline
                -\divG & 0 & \gamma_0 \\
                0 & -\divG & \alpha_1 \pa_t - 2\mu\DeltaG
            \end{array}\right]
            \left[\begin{array}{c}
                \bf_0 \\
                \bf \\
                \hline
                f
            \end{array} \right], \qquad \gamma_0 f := f(0).
        \end{equation}
        The left-hand side of the reformulated system \eqref{eq:reformulations-eq}, in turn, is obtained by applying $\cD$ to the left-hand side $\cA x$ of the augmented problem \eqref{eq:stokes-type-aug}:
        \begin{equation} \label{eq:DA-formal}
            \begin{split}
            \cD (\cA x) & =
            \cD
            \left[\begin{array}{c}
                \bu(0) + \nablaG u_0 \\
                \alpha_1 \pa_t \bu + \alpha_2 \bu \cdot \nablaG \bu - 2 \mu \DeltaS \bu + \nablaG u \\
                \hline
                \divG \bu
            \end{array}\right] \\
            & =
            \left[\begin{array}{c}
                \bu(0) + \nablaG u_0 \\
                \alpha_1 \pa_t \bu + \alpha_2 \bu \cdot \nablaG \bu - 2 \mu \DeltaS \bu + \nablaG u \\
                \hline
                (\divG \bu)(0) - \divG \big( \bu(0) + \nablaG u_0 \big) \\
                (\alpha_1 \pa_t - 2 \mu \DeltaG) \divG \bu - \divG \big( \alpha_1 \pa_t \bu + \alpha_2 \bu \cdot \nablaG \bu - 2 \mu \DeltaS \bu + \nablaG u \big)
            \end{array}\right].
            \end{split}
        \end{equation}
        The first two components are those of $\cA x$. The last two simplify by means of four calculus rules:
        \begin{equation} \label{eq:calculus-rules}
            \begin{aligned}
                & \text{(R1)} \quad \divG \nablaG = \DeltaG, & \qquad & \text{(R2)} \quad \divG \DeltaS (\cdot) = \DeltaG \divG (\cdot) + \divG (\bR \cdot), \\
                & \text{(R3)} \quad \divG \pa_t = \pa_t \divG, & \qquad & \text{(R4)} \quad (\divG \bv)(0) = \divG (\bv(0)).
            \end{aligned}
        \end{equation}
        Rule (R1) is the definition of $\DeltaG$, and (R2) is, formally, identity \eqref{eq:stokes-iden-strong}. Rules (R3) and (R4) hold because $\G$ is stationary. Rigorous versions for $\G$ of class $C^2$ and $C^3$ are proven in Lemma~\ref{lem:calculus} (calculus rules). By (R4) and (R1), the third component of the composition \eqref{eq:DA-formal} equals
        \begin{equation*}
            (\divG \bu)(0) - \divG \big( \bu(0) + \nablaG u_0 \big) = - \DeltaG u_0.
        \end{equation*}
        In the fourth component of \eqref{eq:DA-formal}, the terms $\pm \alpha_1 \pa_t \divG \bu$ cancel by (R3). By (R2), the terms involving $\DeltaS \bu$ and $\DeltaG \divG \bu$ reduce to $2 \mu \, \divG (\bR \bu)$, yielding
        \begin{equation*}
            \begin{split}
                & \alpha_1 \big( \pa_t \divG \bu - \divG \pa_t \bu \big) + 2 \mu \big( \divG \DeltaS \bu - \DeltaG \divG \bu \big) - \divG \nablaG u - \alpha_2 \divG (\bu \cdot \nablaG \bu) \\
                & \qquad = 2 \mu \, \divG (\bR \bu) - \DeltaG u - \alpha_2 \divG (\bu \cdot \nablaG \bu).
            \end{split}
        \end{equation*}
        Hence, $\cD(\cA x)$ is precisely the left-hand side of the reformulated system \eqref{eq:reformulations-eq}. We denote it by $\cB x$. The cancellation is no accident. Together, the rules (R2) and (R3) yield, for tangential fields $\bv$, the commutator identity
        \begin{equation} \label{eq:commutator}
            \divG \big( (\alpha_1 \pa_t - 2 \mu \DeltaS) \bv \big) = (\alpha_1 \pa_t - 2 \mu \DeltaG) \divG \bv - 2 \mu \, \divG (\bR \bv).
        \end{equation}
        Thus, the tangential divergence turns the vector operator $\alpha_1 \pa_t - 2 \mu \DeltaS$ into its scalar counterpart $\alpha_1 \pa_t - 2 \mu \DeltaG$, up to the curvature term. The data map applies precisely this scalar counterpart to the divergence equation. Therefore, the two terms in which the scalar counterpart acts on $\divG \bu$ cancel, and the velocity enters the pressure Poisson equation only through the curvature term $2 \mu \, \divG (\bR \bu)$ and the convective term. Since $\cB x = \cD (\cA x)$, the reformulation of $\cA x = y$ is $\cB x = \cD y$. With general data $z$, it reads
        \begin{equation} \label{eq:reformulations-comp}
            \cB x = z.
        \end{equation}
        General data arise naturally. With data of the form $z = \cD y$, the reformulation is solvable, since every solution of the primal problem solves it. In practice, however, the data of the reformulation are perturbed, e.g., by discretization, and thus need not be of the form $\cD y$. It is therefore not obvious at first glance that the reformulated problem $\cB x = z$ is solvable for every $z$, let alone stably. If the primal problem is solvable for all data $y$, then this question reduces to whether every $z$ is of the form $\cD y$, i.e., whether the data map $\cD$ is onto. As we show next, the structure of $\cD$ makes it invertible.

        We point out two salient features of the data map \eqref{eq:data-map-formal-full}. First, $\cD$ is linear and independent of $\alpha_2$: the convective term affects $\cA$ and $\cB$, but not the data map. Second, $\cD$ is \textit{block triangular}, i.e., it can be written as 
        \begin{equation} \label{eq:cD-y1-y2}
            \cD(y_1, y_2) = (y_1, M y_1 + L y_2),
        \end{equation}
        where $y_1 = (\bf_0, \bf)$ collects the initial and momentum data, and $y_2 = f$ is the divergence datum. The off-diagonal block $M$ acts on the momentum (and initial) data, and the diagonal block $L$ acts on the divergence datum:
        \begin{equation} \label{eq:operators-K-L}
            M y_1 = (-\divG \bf_0, - \divG \bf), \qquad
            L y_2 =
            \begin{cases}
                - 2 \mu \DeltaG f & \text{ if } \alpha_1 = 0, \\
                (f(0), \pa_t f - 2 \mu \DeltaG f) & \text{ if } \alpha_1 = 1.
            \end{cases}
        \end{equation}
        The diagonal block $L$ is responsible for the reversibility of the reformulation. Indeed, let $\cB x = \cD y$. Since $\cB x = \cD (\cA x)$ by definition and $\cD$ is linear, we have $\cD (\cA x - y) = 0$. We write $\cA x - y = (r_1, r_2)$, where $r_1$ collects the residuals of the initial condition and the momentum equation, and $r_2 = \divG \bu - f$ is the residual of the divergence equation. By the block structure \eqref{eq:cD-y1-y2},
        \begin{equation*}
            (0, 0) = \cD (r_1, r_2) = (r_1, M r_1 + L r_2).
        \end{equation*}
        The first component gives $r_1 = 0$, i.e., the initial condition and the momentum equation hold. The second component then reduces to $L r_2 = L (\divG \bu - f) = 0$. Hence, the divergence equation is recovered, because the Poisson ($\alpha_1 = 0$) or heat ($\alpha_1 = 1$) problem with zero data has only the trivial solution. In the same way, the invertibility of $L$ makes $\cD$ invertible.

        \textit{Strategy}. If $\cD$ is a linear homeomorphism, then the identity
        \begin{equation} \label{eq:B=DA}
            \cB = \cD \circ \cA
        \end{equation}
        is equivalent to $\cA = \cD^{-1} \circ \cB$. Consequently, $\cA x = y$ if and only if $\cB x = \cD y$. Moreover, $\cB$ is a homeomorphism if and only if $\cA$ is; see Section~\ref{subsec:abstract-reformulation}. For each of the three problems, it thus remains to
        \begin{enumerate}
            \item[(a)] choose spaces $\sX$, $\sY$, $\sZ$ such that $\cA : \sX \to \sY$ and $\cB : \sX \to \sZ$ are well defined, the primal problem can be expected to be well-posed, and the data map $\cD : \sY \to \sZ$ is a linear homeomorphism. By the triangular structure of $\cD$, the latter amounts to the invertibility of $L$, which is a Poisson or a heat operator.
            \item[(b)] verify the factorization \eqref{eq:B=DA} in these spaces. This verification amounts to justifying the rules (R1)--(R4) at the regularity at hand.
        \end{enumerate}
        Once item (a) is settled, item (b) is a matter of rigorous calculus. Item (a) is the one that requires some care.

        \textit{Functional setting}. The data map differentiates the divergence datum: it contains $-2 \mu \DeltaG f$ and, in the transient case, $\pa_t f$ and $f(0)$. We call the pressure Poisson equation and, if $\alpha_1 = 1$, the initial Poisson equation the \textit{pressure equations}. The space of divergence data is dictated by where $\divG \bu$ lands for an admissible velocity $\bu$. It then remains to check that $L$ maps this space onto the space of data of the pressure equations. The latter is dictated by $-\divG$ of the momentum data and by $-\DeltaG$ of the pressure.
        \begin{itemize}
            \item \textit{TS}. In the energy setting $(\bu, u) \in \bH^1_{t, \perp} \times L^2_\#$, we have $\divG \bu \in L^2_\#$, so the divergence data lie in $L^2_\#$. The Poisson operator $-2 \mu \DeltaG$ maps $L^2_\#$ onto $H^{-2}_\#$. This is consistent with the rest: the momentum datum lies in $\bH^{-1}_{t, \perp}$, and both $\divG \bf$ and $\DeltaG u$ lie in $H^{-2}_\#$. The price to pay is that the pressure Poisson equation holds only ultraweakly, i.e., it is tested with $H^2_\#$-functions. An $H^1$-based setting is recovered a posteriori if its datum lies in $H^{-1}_\#$; see Remark~\ref{rmk:equal-order} (equal-order formulations).
            \item \textit{TTS}. In the energy setting, the velocity satisfies $\bu \in L^2(I; \bH^1_{t, \perp}) \cap H^1(I; \bH^{-1}_{t, \perp})$ and the pressure satisfies $u \in L^2(I; L^2_\#)$. Applying $\divG$ to $\bu$ gives $\divG \bu \in L^2(I; L^2_\#) \cap H^1(I; H^{-2}_\#) = \cW^0_\#$. Hence, the divergence data lie in $\cW^0_\#$; in particular, \textit{the divergence datum must possess a time derivative}. The rest is consistent. The heat operator $f \mapsto (f(0), \pa_t f - 2 \mu \DeltaG f)$ maps $\cW^0_\#$, the solution space of the ultraweak heat equation, onto $H^{-1}_\# \times \cV^{-2}_\#$. These are exactly the spaces in which the data of the pressure equations lie, since the momentum datum lies in $\bcV^{-1}_{t, \perp}$ and the initial datum in $\bL^2_{t, \perp}$. Two points require care. First, the well-posedness of TTS with divergence data in $\cW^0_\#$ is not standard; see Section~\ref{subsec:main-results}. Second, since $f(0) \in H^{-1}_\#$, the compatibility condition $\divG \bf_0 = f(0)$ is an identity in $H^{-1}_\#$ that constrains the data. The Lagrange multiplier $u_0 \in H^1_\#$ in the relaxed initial condition \eqref{eq:new-ic} removes this constraint.
            \item \textit{TTNS}. If $\bu$ is not solenoidal, then the convective term is not controlled by the energy estimates; see Section~\ref{subsec:reform-of-ns}. This forces a strong setting, in which everything shifts up by one order of spatial regularity. Thus, the velocity lies in $\bcW^2_t$, the divergence datum lies in $\cW^1_\#$, and the pressure equations have data in $L^2_\# \times \cV^{-1}_\#$. The heat equation for $f$ is understood weakly. 
        \end{itemize}
        These settings are summarized in Table~\ref{tab:settings} of Section~\ref{subsec:settings}.

        \begin{remark}[Killing fields] \label{rmk:killing}
            The three problems require different treatments of the Killing fields $\cK$ from Section~\ref{subsec:killing}.
            \begin{itemize}
                \item \textit{TS}. Since $\DG \bv_\cK = \bzero$ and $\divG \bv_\cK = 0$ for all $\bv_\cK \in \cK$, the stationary problem determines the velocity only modulo $\cK$. Moreover, it is solvable only if $\bf \perp \cK$. We therefore pose it on the spaces $\bH^k_{t, \perp}$ in Section~\ref{subsec:reform-of-stokes}.
                \item \textit{TTS}. Suppose that TTS is posed on the full spaces $\bH^k_t$. We test with $\bw_\cK \in \cK$ and use $\DG \bw_\cK = \bzero$, $\divG \bw_\cK = 0$, and $(\nablaG u_0, \bw_\cK) = 0$. This test shows that the Killing component of the velocity decouples and solves the ODE
                \begin{equation*}
                    \frac{\d}{\dt} (\bu, \bw_\cK) = \langle \bf, \bw_\cK \rangle_{-1,1}, \qquad (\bu(0), \bw_\cK) = (\bf_0, \bw_\cK).
                \end{equation*}
                By linearity, $\bu - \cP_\cK \bu$ solves the problem with data orthogonal to $\cK$. Thus, posing TTS on $\bH^k_{t, \perp}$ only removes an explicitly solvable finite-dimensional ODE.
                This ODE is also responsible for the dependence of the stability constants on $T$ on the full spaces $\bH^k_t$, as the following example shows: $\bu = t \bw_\cK$ solves the problem with $\bf = \bw_\cK$, $\bf_0 = \bzero$, $f = 0$.
                We thus formulate TTS in the spaces $\bcV^k_{t, \perp}, \bcW^k_{t, \perp}$ in Section~\ref{subsec:reform-of-ts}.
                \item \textit{TTNS}. For $\bv_\cK \in \cK$, integration by parts and the skew-symmetry of $\nablaG \bv_\cK$ give
                \begin{equation*}
                    (\bu \cdot \nablaG \bu, \bv_\cK) = - ((\nablaG \bv_\cK) \bu, \bu) - (\divG \bu, \bu \cdot \bv_\cK) = - (\divG \bu, \bu \cdot \bv_\cK).
                \end{equation*}
                The right-hand side does not vanish in general if $\divG \bu \neq 0$. In other words, \textit{even if the initial and forcing data are Killing-free, the convective term can still generate nontrivial Killing modes}. Posing the momentum equation on $\bH^k_{t, \perp}$ would replace $\bu \cdot \nablaG \bu$ by $\cP_\cK^\perp (\bu \cdot \nablaG \bu)$, thereby changing the model. We therefore pose TTNS on the full spaces $\bH^k_t$. This choice is possible because $\pa_t \bu$ controls the Killing component of the velocity. We refer to Section~\ref{subsec:reform-of-ns}.
            \end{itemize}
        \end{remark}

    \subsection{Abstract framework} \label{subsec:abstract-reformulation}
        The following statement makes the strategy of Section~\ref{subsec:reform-of-stokes-type-pbs} precise in an abstract setting. Let $\sX$, $\sY$, $\sZ$ be Banach spaces.
        \begin{proposition}[transfer of well-posedness] \label{prop:transf-of-well-pos}
            Let $\cD : \sY \to \sZ$ be a linear homeomorphism, and let $\cA : \sX \to \sY$ and $\cB : \sX \to \sZ$ be (possibly nonlinear) maps such that $\cB = \cD \circ \cA$:
            \[
                \sX \xrightarrow{\ \ \cA \ \ } \sY \xrightarrow{\ \ \cD \ \ } \sZ, \qquad \cB = \cD \circ \cA.
            \]
            Then the following statements are valid.
            \begin{enumerate}
                \item[(i)] \textit{Equivalence}. For all $(x, y) \in \sX \times \sY$, we have $\cA x = y$ if and only if $\cB x = \cD y$.
                \item[(ii)] \textit{Well-posedness}. The map $\cB$ is a homeomorphism if and only if $\cA$ is. In this case, $\cB^{-1} = \cA^{-1} \circ \cD^{-1}$. In other words, for every $z \in \sZ$, the unique solution $x \in \sX$ of $\cB x = z$ is the unique solution of the primal problem $\cA x = \cD^{-1} z$.
                \item[(iii)] \textit{Bounds}. If $\cA$ is a linear homeomorphism, then so is $\cB$, and
                \begin{equation} \label{eq:cB-bounds}
                    \|\cB\| \leq \|\cD\| \|\cA\|, \qquad \|\cB^{-1}\| \leq \|\cA^{-1}\| \|\cD^{-1}\|.
                \end{equation}
            \end{enumerate}
        \end{proposition}
        \begin{proof}
            (i) Apply $\cD$, resp.\ $\cD^{-1}$, to both sides. (ii) Compositions of homeomorphisms are homeomorphisms. (iii) Use $\|S T\| \leq \|S\| \|T\|$ for bounded linear operators $S$ and $T$.
        \end{proof}
        We stress that statement (i) requires nothing of $\cA$; we use this in Section~\ref{sec:conclusion}. The data maps of Section~\ref{subsec:reform-of-stokes-type-pbs} are block triangular, so that their invertibility reduces to that of their diagonal block. To formalize this, let $\sY_1$, $\sY_2$, $\sZ_2$ be Banach spaces. Since $\cD$ preserves the first component, we set $\sY := \sY_1 \times \sY_2$ and $\sZ := \sY_1 \times \sZ_2$, both endowed with the product norm $\|(a, b)\| := (\|a\|^2 + \|b\|^2)^{1/2}$.
        \begin{lemma}[triangular data maps] \label{lem:triangular}
            Let $M : \sY_1 \to \sZ_2$ be a bounded linear operator and $L : \sY_2 \to \sZ_2$ be a linear homeomorphism. Then
            \begin{equation*}
                \cD :=
                \left[ \begin{array}{cc} \operatorname{id} & 0 \\ M & L \end{array} \right] : \sY \to \sZ
            \end{equation*}
            is a linear homeomorphism with inverse
            \begin{equation*}
                \cD^{-1} =
                \left[ \begin{array}{cc} \operatorname{id} & 0 \\ - L^{-1} M & L^{-1} \end{array} \right] : \sZ \to \sY,
            \end{equation*}
            and
            \begin{equation} \label{eq:triangular-bounds}
                \|\cD\| \leq 1 + \|M\| + \|L\|, \qquad \|\cD^{-1}\| \leq 1 + \|L^{-1}\| (1 + \|M\|).
            \end{equation}
            Consequently, let $\cA$ and $\cB$ be as in Proposition~\ref{prop:transf-of-well-pos} with this $\cD$, and let $\cA$ be a homeomorphism. Then, for every $z = (z_1, z_2) \in \sZ$, the unique solution $x \in \sX$ of $\cB x = z$ is the unique solution of
            \begin{equation} \label{eq:reform-as-primal}
                \cA x = \big( z_1, L^{-1} (z_2 - M z_1) \big).
            \end{equation}
        \end{lemma}
        \begin{proof}
            Solving $\cD y = z$ is block forward substitution. The first block row gives $y_1 = z_1$, and the second one then reads $L y_2 = z_2 - M z_1$. The bounds \eqref{eq:triangular-bounds} follow from the triangle inequality and $\max\{\|a\|, \|b\|\} \leq \|(a, b)\| \leq \|a\| + \|b\|$. Finally, \eqref{eq:reform-as-primal} is statement (ii) of Proposition~\ref{prop:transf-of-well-pos}.
        \end{proof}
        \begin{remark}[interpretation] \label{rmk:factorization}
            Statement (i) of Proposition~\ref{prop:transf-of-well-pos} says that the reformulation is both \textit{valid} and \textit{reversible}: every solution of the primal problem solves the reformulation, and vice versa. No information is lost, and reversibility requires no separate argument. Lemma~\ref{lem:triangular} (triangular data maps) says that solving the reformulation with arbitrary data $z$ amounts to solving the primal problem with a modified divergence datum. The momentum (and initial) data are unchanged, and the divergence datum $f$ solves the single linear auxiliary problem $L f = z_2 - M z_1$. This auxiliary problem is a Poisson problem for TS and a heat problem for TTS and TTNS. In particular, perturbing the data of the pressure Poisson equation is equivalent to perturbing the divergence constraint. Finally, the operator $\cD$ in our case is always linear, while the operator $\cA$, and hence $\cB$, may be nonlinear, as is the case for TTNS.
        \end{remark}

    \subsection{Functional settings and data maps} \label{subsec:settings}
        We now make the functional settings precise. They are summarized in Table~\ref{tab:settings}. In all three cases, $\cD$ has the triangular form of Lemma~\ref{lem:triangular} (triangular data maps) with $M = - \divG$, applied componentwise (i.e., $(-\divG \bf_0, -\divG \bf)$). The diagonal block $L$ is either the Poisson operator $- 2 \mu \DeltaG$ or, for $r \in \{0,1\}$, the heat operator
        \begin{equation} \label{eq:Lr}
            L_r : \cW^r_\# \to H^{r-1}_\# \times \cV^{r-2}_\#, \qquad L_r f := \big( f(0), \pa_t f - 2 \mu \DeltaG f \big).
        \end{equation}
        Here, $\DeltaG$ is understood ultraweakly for $r = 0$ and weakly for $r = 1$.
        \begin{table}[htbp]
            \caption{Functional settings of Section~\ref{subsec:settings}. For TTS and TTNS, the momentum data include the initial datum, and the pressure data include the datum of the initial Poisson equation. The last column is the diagonal block of the data map; see \eqref{eq:Lr}.} \label{tab:settings}
            \centering
            \resizebox{\textwidth}{!}{
            \begin{tabular}{@{}llllll@{}}
                \hline
                & solution & momentum data & divergence data & pressure data & data map \\
                & $\sX$ & $\sY_1$ & $\sY_2$ & $\sZ_2$ & $L$ \\
                \hline
                TS & $(\bu, u) \in \bH^1_{t, \perp} \times L^2_\#$ & $\bf \in \bH^{-1}_{t, \perp}$ & $f \in L^2_\#$ & $g \in H^{-2}_\#$ & $- 2 \mu \DeltaG$ \\
                TTS & $(u_0, \bu, u) \in H^1_\# \times \bcW^1_{t, \perp} \times \cV^0_\#$ & $(\bf_0, \bf) \in \bL^2_{t, \perp} \times \bcV^{-1}_{t, \perp}$ & $f \in \cW^0_\#$ & $(g_0, g) \in H^{-1}_\# \times \cV^{-2}_\#$ & $L_0$ \\
                TTNS & $(u_0, \bu, u) \in H^2_\# \times \bcW^2_t \times \cV^1_\#$ & $(\bf_0, \bf) \in \bH^1_t \times \bcV^0_t$ & $f \in \cW^1_\#$ & $(g_0, g) \in L^2_\# \times \cV^{-1}_\#$ & $L_1$ \\
                \hline
            \end{tabular}}
        \end{table}

        We first fix the meaning of the differential operators in negative spaces. The Laplace--Beltrami operator $\DeltaG$ is understood as in Section~\ref{subsec:equiv-norms}, i.e., strongly on $H^2_\#$, weakly on $H^1_\#$, and ultraweakly on $L^2_\#$. These definitions are consistent with each other. The surface diffusion operator is understood strongly on $\bH^2_t$ and weakly on $\bH^1_{t, \perp}$; see Section~\ref{subsec:equiv-norms}. In the latter case, $\langle - \DeltaS \bv, \bw \rangle_{-1,1} := (\DG \bv, \DG \bw)$ for $\bw \in \bH^1_{t, \perp}$. Gradient and divergence are extended to negative spaces by integration by parts. Recall from Section~\ref{subsec:killing} that gradients are Killing-free: $\nablaG$ maps $H^{s+1}_\#$ into $\bH^s_{t, \perp}$ for $s \in \{0, 1\}$. Hence, for $\bw \in \bL^2_t$ or $\bw \in \bH^{-1}_{t, \perp}$, and for $u \in L^2_\#$, we may define
        \begin{equation} \label{eq:grad-div}
            \langle \divG \bw, \phi \rangle := - \langle \bw, \nablaG \phi \rangle, \qquad \langle \nablaG u, \bv \rangle_{-1,1} := - (u, \divG \bv),
        \end{equation}
        where $\phi \in H^1_\#$ if $\bw \in \bL^2_t$, $\phi \in H^2_\#$ if $\bw \in \bH^{-1}_{t, \perp}$, and $\bv \in \bH^1_{t, \perp}$. By the divergence theorem on the closed hypersurface $\G$, these operators extend the classical ones. In particular, $\divG \bw = \tr(\nablaG \bw)$ has vanishing mean for $\bw \in \bH^1_t$. All of these operators are bounded, with constants depending only on $\G$. Thus, $\divG$ maps $\bL^2_t$ into $H^{-1}_\#$ and $\bH^{-1}_{t, \perp}$ into $H^{-2}_\#$. It also maps $\bH^1_t$ into $L^2_\#$ and, for $\G$ of class $C^3$, $\bH^2_t$ into $H^1_\#$. Finally, $\nablaG$ maps $L^2_\#$ into $\bH^{-1}_{t, \perp}$. Since $\G$ is stationary, all of these operators act pointwise in time on Bochner spaces.

        \subsubsection{Tangent Stokes} \label{subsec:reform-of-stokes}
            Let $\sX := \bH^1_{t, \perp} \times L^2_\#$ be the solution space, $\sY := \bH^{-1}_{t, \perp} \times L^2_\#$ the space of primal data, and $\sZ := \bH^{-1}_{t, \perp} \times H^{-2}_\#$ the space of data of the reformulation. The primal problem reads: given $(\bf, f) \in \sY$, find $(\bu, u) \in \sX$ such that
            \begin{equation*}
                \cA (\bu, u) := \big( - 2 \mu \DeltaS \bu + \nablaG u, \ \divG \bu \big) = (\bf, f),
            \end{equation*}
            i.e., for all $(\bv, v) \in \bH^1_{t, \perp} \times L^2_\#$,
            \begin{equation} \label{eq:stokes-lp}
                \begin{cases}
                    2 \mu (\DG \bu, \DG \bv) - (u, \divG \bv) \kern-0.72em & = \langle \bf, \bv \rangle_{-1,1}, \\
                    (\divG \bu, v) \kern-0.72em & = (f, v).
                \end{cases}
            \end{equation}
            The reformulation with general data reads: given $(\bg, g) \in \sZ$, find $(\bu, u) \in \sX$ such that
            \begin{equation*}
                \cB (\bu, u) := \big( - 2 \mu \DeltaS \bu + \nablaG u, \ - \DeltaG u + 2 \mu \, \divG (\bR \bu) \big) = (\bg, g),
            \end{equation*}
            i.e., for all $(\bv, \phi) \in \bH^1_{t, \perp} \times H^2_\#$,
            \begin{equation} \label{eq:stokes-lp-reform}
                \begin{cases}
                    2 \mu (\DG \bu, \DG \bv) - (u, \divG \bv) \kern-0.72em & = \langle \bg, \bv \rangle_{-1,1}, \\
                    - (u, \DeltaG \phi) - 2 \mu (\bR \bu, \nablaG \phi) \kern-0.72em & = \langle g, \phi \rangle_{-2,2}.
                \end{cases}
            \end{equation}
            The second equation is the ultraweak form of $- \DeltaG u + 2 \mu \, \divG (\bR \bu) = g$. The data map is the formal one \eqref{eq:data-map-formal-full} with $\alpha_1 = 0$, i.e., $\cD (\bf, f) := (\bf, - \divG \bf - 2 \mu \DeltaG f)$. By the isometries of Section~\ref{subsec:equiv-norms} for $\Lambda$ and $\Lambda_S$ and by the boundedness of $\nablaG$ and $\divG$, the operators $\cA : \sX \to \sY$, $\cB : \sX \to \sZ$, and $\cD : \sY \to \sZ$ are bounded linear operators. For $\cB$, we note that $\bR \bu \in \bL^2_t$, whence $\divG (\bR \bu) \in H^{-1}_\# \subseteq H^{-2}_\#$.

        \subsubsection{Transient tangent Stokes} \label{subsec:reform-of-ts}
            Let $\sX := H^1_\# \times \bcW^1_{t, \perp} \times \cV^0_\#$ be the solution space, $\sY := [\bL^2_{t, \perp} \times \bcV^{-1}_{t, \perp}] \times \cW^0_\#$ the space of primal data, and $\sZ := [\bL^2_{t, \perp} \times \bcV^{-1}_{t, \perp}] \times [H^{-1}_\# \times \cV^{-2}_\#]$ the space of data of the reformulation. Recall from \eqref{eq:spaces-cVr-cWr} that $\bcW^1_{t, \perp} = L^2(I; \bH^1_{t, \perp}) \cap H^1(I; \bH^{-1}_{t, \perp})$ and $\cW^0_\# = L^2(I; L^2_\#) \cap H^1(I; H^{-2}_\#)$. We endow these spaces with the norms
            \begin{equation} \label{eq:ts-norms}
                \begin{split}
                    \|x\|_\sX^2 & = \|u_0\|_1^2 + \|\bu(0)\|^2 + \int_I \|\pa_t \bu\|_{-1}^2 + \|\bu\|_1^2 + \|u\|^2 \dt, \\
                    \|y\|_\sY^2 & = \|\bf_0\|^2 + \|f(0)\|_{-1}^2 + \int_I \|\bf\|_{-1}^2 + \|f\|^2 + \|\pa_t f\|_{-2}^2 \dt, \\
                    \|z\|_\sZ^2 & = \|\bg_0\|^2 + \|g_0\|_{-1}^2 + \int_I \|\bg\|_{-1}^2 + \|g\|_{-2}^2 \dt.
                \end{split}
            \end{equation}
            The primal problem reads: given $y = (\bf_0, \bf, f) \in \sY$, find $x = (u_0, \bu, u) \in \sX$ such that
            \begin{equation} \label{eq:ts-weak}
                \begin{cases}
                    \begin{aligned}
                        \bu(0) + \nablaG u_0 & = \bf_0 & \text{in } \bL^2_{t, \perp}, \\
                        \pa_t \bu - 2 \mu \DeltaS \bu + \nablaG u & = \bf & \text{in } \bcV^{-1}_{t, \perp}, \\
                        \divG \bu & = f & \text{in } \cW^0_\#,
                    \end{aligned}
                \end{cases}
            \end{equation}
            which we write as $\cA x = y$. We stress that the divergence equation is posed in $\cW^0_\#$ rather than in $\cV^0_\#$, because the heat operator $L_0$ in the data map differentiates the divergence datum in time. The space $\cW^0_\#$ of divergence data is also the nonstandard feature of Proposition~\ref{prop:primal} (primal problems) below.
            
            Moreover, the divergence equation in \eqref{eq:ts-weak} is hard to discretize, because its test space is the dual $(\cW^0_\#)'$, which has no simple description; see Section~\ref{subsec:boch-sobolev-spaces}. The reformulation removes this difficulty. The data map replaces the divergence equation with the pressure equations, which are posed in $H^{-1}_\# \times \cV^{-2}_\#$ and tested with $H^1_\# \times \cV^2_\#$. If the datum of the pressure Poisson equation lies in $\cV^{-1}_\#$, the test space $\cV^2_\#$ can be relaxed to $\cV^1_\#$, so that both pressure equations are tested with $H^1$-based spaces; see Remark~\ref{rmk:equal-order} (equal-order formulations).
            
            By Remark~\ref{rmk:lm-vector} (vector-valued counterparts of Lions--Magenes results), we have $\bu(0) + \nablaG u_0 \in \bL^2_{t, \perp}$. Together with $\divG \bu \in \cW^0_\#$ from Lemma~\ref{lem:calculus} (calculus rules), this shows that $\cA : \sX \to \sY$ is a bounded linear operator. Similar spaces appear in \cite[Theorem 3.7]{GuberovicSchwabStevenson2014} for the transient Stokes problem in flat domains with the free-slip boundary condition. We stress that posing the problem on $\bH^k_{t, \perp}$ entails no loss of generality; see Remark~\ref{rmk:killing} (Killing fields). The reformulation with general data reads: given $z = (\bg_0, \bg, g_0, g) \in \sZ$, find $x = (u_0, \bu, u) \in \sX$ such that
            \begin{equation} \label{eq:ts-reform}
                \begin{cases}
                \begin{aligned}
                    \bu(0) + \nablaG u_0 & = \bg_0 & \text{in } \bL^2_{t, \perp}, \\
                    \pa_t \bu - 2 \mu \DeltaS \bu + \nablaG u & = \bg & \text{in } \bcV^{-1}_{t, \perp}, \\
                    -\DeltaG u_0  & = g_0 & \text{in } H^{-1}_\#, \\
                    - \DeltaG u + 2 \mu \, \divG (\bR \bu) & = g & \text{in } \cV^{-2}_\#,
                \end{aligned}
                \end{cases}
            \end{equation}
            which we write as $\cB x = z$. The operator $\cB : \sX \to \sZ$ is bounded and linear. The first two equations are those of the primal problem \eqref{eq:ts-weak}. The third one is the weak form of the initial Poisson equation \eqref{eq:new-ic-pressure}. The fourth one is the ultraweak form of the pressure Poisson equation \eqref{eq:intro-pressure-poi-copy}, i.e.,
            \begin{equation*}
                \int_I - (u, \DeltaG \phi) - 2 \mu (\bR \bu, \nablaG \phi) \dt = \int_I \langle g, \phi \rangle_{-2,2} \dt, \qquad \forall \phi \in \cV^2_\#.
            \end{equation*}
            The data map is the formal one \eqref{eq:data-map-formal-full} with $\alpha_1 = 1$. It has the triangular form of Lemma~\ref{lem:triangular} (triangular data maps) with $M (\bf_0, \bf) := (- \divG \bf_0, - \divG \bf)$ and with the ultraweak heat operator $L := L_0$ from \eqref{eq:Lr}.

        \subsubsection{Transient tangent Navier--Stokes} \label{subsec:reform-of-ns}
            Let $d = 2$ and $\G$ be of class $C^3$. Since we do not assume that $\bu$ is solenoidal, the convective term cannot be controlled in the energy setting of Section~\ref{subsec:reform-of-ts}. Indeed, suppose that $\bu \in \bcW^1_t$ and that the divergence equation is posed in $\cW^0_\#$. For $\bv \in \bcV^1_t$, integration by parts gives
            \[
                (\bu \cdot \nablaG \bu, \bv) = - (\bu \cdot \nablaG \bv, \bu) - (\divG \bu, \bu \cdot \bv).
            \]
            By Ladyzhenskaya's inequality, the first term defines an element of $\bcV^{-1}_t$, as in the solenoidal case. The second term vanishes for solenoidal velocities. Otherwise, it is only bounded by $\|\divG \bu\| \|\bu\|_{\bL^4(\G)} \|\bv\|_{\bH^1(\G)}$, where $\|\divG \bu\| \in L^2(I)$ and $\|\bu\|_{\bL^4(\G)} \in L^4(I)$. This bound only places $(\divG \bu)\bu$ in $L^{4/3}(I; \bH^{-1}_t)$, short of $\bcV^{-1}_t$. If instead $\bu \in \bcW^2_t$, then $\bu \cdot \nablaG \bu \in \bcV^0_t$. This motivates the setting of strong solutions. In view of Remark~\ref{rmk:killing} (Killing fields), we pose the problem on the full spaces $\bH^k_t$. Namely, we let $\sX := H^2_\# \times \bcW^2_t \times \cV^1_\#$ be the solution space, $\sY := [\bH^1_t \times \bcV^0_t] \times \cW^1_\#$ the space of primal data, and $\sZ := [\bH^1_t \times \bcV^0_t] \times [L^2_\# \times \cV^{-1}_\#]$ the space of data of the reformulation. Recall that $\bcW^2_t = L^2(I; \bH^2_t) \cap H^1(I; \bL^2_t)$ and $\cW^1_\# = L^2(I; H^1_\#) \cap H^1(I; H^{-1}_\#)$. We endow $\bcW^2_t$ with the norm analogous to the $\cW$-norms \eqref{eq:cW-norms}. The primal problem reads: given $y = (\bf_0, \bf, f) \in \sY$, find $x = (u_0, \bu, u) \in \sX$ such that
            \begin{equation} \label{eq:tns}
                \begin{cases}
                \begin{aligned}
                    \bu(0) + \nablaG u_0 & = \bf_0 & \text{in } \bH^1_t, \\
                    \pa_t \bu + \bu \cdot \nablaG \bu - 2 \mu \DeltaS \bu + \nablaG u & = \bf & \text{in } \bcV^0_t, \\
                    \divG \bu & = f & \text{in } \cW^1_\#,
                \end{aligned}
                \end{cases}
            \end{equation}
            which we write as $\cA x = y$. Here, $\divG \bu \in \cW^1_\#$ by Lemma~\ref{lem:calculus} (calculus rules) below, and $\cA : \sX \to \sY$ is a nonlinear operator. The reformulation with general data reads: given $z = (\bg_0, \bg, g_0, g) \in \sZ$, find $x = (u_0, \bu, u) \in \sX$ such that
            \begin{equation} \label{eq:tns-reform}
                \begin{cases}
                \begin{aligned}
                    \bu(0) + \nablaG u_0 & = \bg_0 & \text{in } \bH^1_t, \\
                    \pa_t \bu + \bu \cdot \nablaG \bu - 2 \mu \DeltaS \bu + \nablaG u & = \bg & \text{in } \bcV^0_t, \\
                    - \DeltaG u_0 & = g_0 & \text{in } L^2_\#, \\
                    - \DeltaG u - \divG (\bu \cdot \nablaG \bu) + 2 \mu \, \divG (\bR \bu) & = g & \text{in } \cV^{-1}_\#,
                \end{aligned}
                \end{cases}
            \end{equation}
            which we write as $\cB x = z$, with $\cB : \sX \to \sZ$ a nonlinear operator. The last equation is $H^1$-conforming as it stands. In weak form, this equation reads
            \begin{equation*}
                \int_I (\nablaG u, \nablaG \phi) + (\bu \cdot \nablaG \bu, \nablaG \phi) - 2 \mu (\bR \bu, \nablaG \phi) \dt = \int_I \langle g, \phi \rangle_{-1,1} \dt, \qquad \forall \phi \in \cV^1_\#.
            \end{equation*}
            The data map is that of TTS, but one order of spatial regularity higher. Namely, the off-diagonal block is $M = (-\divG, - \divG) : \bH^1_t \times \bcV^0_t \to L^2_\# \times \cV^{-1}_\#$ and the diagonal block is the weak heat operator $L := L_1$.

        \subsubsection{Data maps} \label{subsec:data-maps}
            Inverting the data map amounts to recovering the divergence datum from the data of the reformulation. Recovering this datum requires solving a Poisson or a heat equation, and nothing else.
        \begin{lemma}[data maps] \label{lem:data-maps}
            Let $\G$ be of class $C^2$. Then, in each of the three settings of Table~\ref{tab:settings}, $\cD : \sY \to \sZ$ is a linear homeomorphism, with constants depending only on $\G$ and $\mu$ (and not on $T$). Moreover, $\cD^{-1} z = (z_1, f)$. Here, $z_1 = \bg$ (TS) or $z_1 = (\bg_0, \bg)$ (TTS, TTNS) is the momentum (and initial) datum, and $f \in \sY_2$ is the unique solution of
            \begin{equation} \label{eq:f-from-z}
                - 2 \mu \DeltaG f = g + \divG \bg \ \ \text{(TS)}, \qquad
                \begin{cases}
                    \pa_t f - 2 \mu \DeltaG f \kern-0.72em & = g + \divG \bg, \\
                    f(0) \kern-0.72em & = g_0 + \divG \bg_0,
                \end{cases} \ \ \text{(TTS, TTNS)}.
            \end{equation}
        \end{lemma}
        \begin{proof}
            In view of Lemma~\ref{lem:triangular} (triangular data maps), we only need to show that $M = - \divG$ is bounded and that $L$ is a linear homeomorphism. The former was noted after the duality definitions \eqref{eq:grad-div}, pointwise in time in the transient cases. For TS, $L = - 2 \mu \DeltaG = 2 \mu \Lambda : L^2_\# \to H^{-2}_\#$ is a linear homeomorphism, because $\Lambda$ maps $L^2_\#$ isometrically onto $H^{-2}_\#$; see Section~\ref{subsec:equiv-norms}. For TTS and TTNS, $L = L_r$ with $r = 0$ and $r = 1$, respectively, and Proposition~\ref{prop:heat} (heat equation) states precisely that $L_r : \cW^r_\# \to H^{r-1}_\# \times \cV^{r-2}_\#$ is bijective with $C_1 \|L_r f\| \leq \|f\|_{\cW^r_\#} \leq C_2 \|L_r f\|$. Recall here that the $\cW^r_\#$-norm includes the trace at $t = 0$. Finally, the auxiliary problems \eqref{eq:f-from-z} are the equation $L f = z_2 - M z_1$.
        \end{proof}

    \subsection{Factorization} \label{subsec:factorization}
        We now verify $\cB = \cD \circ \cA$ in the three settings. The computation is the formal one following \eqref{eq:DA-formal}. All we need is that the calculus rules \eqref{eq:calculus-rules} hold at the regularity at hand.
        \begin{lemma}[calculus rules] \label{lem:calculus}
            Let $\G$ be of class $C^2$. Let $\nablaG$ and $\divG$ be defined by the duality definitions \eqref{eq:grad-div}. Let $\DeltaG$ be understood ultraweakly on $L^2_\#$, weakly on $H^1_\#$, and strongly on $H^2_\#$. Let $\DeltaS$ be understood weakly on $\bH^1_{t, \perp}$ and strongly on $\bH^2_t$. Then the rules \eqref{eq:calculus-rules} are valid in the following sense.
            \begin{enumerate}
                \item[(R1)] Let $s \in \{0, 1, 2\}$ and $u \in H^s_\#$, with $H^0_\# = L^2_\#$. Then the identity $\divG \nablaG u = \DeltaG u$ holds in $H^{s-2}_\#$.
                \item[(R2)] Let $\bv \in \bH^1_{t, \perp}$. Then the identity $\divG \DeltaS \bv = \DeltaG \divG \bv + \divG (\bR \bv)$ holds in $H^{-2}_\#$. If $\G$ is of class $C^3$ and $\bv \in \bH^2_t$, then this identity holds in $H^{-1}_\#$.
                \item[(R3)] Let $\bv \in \bcW^1_{t, \perp}$. Then the identity $\divG \pa_t \bv = \pa_t \divG \bv$ holds in $\cV^{-2}_\#$. If $\bv \in \bcW^2_t$, then this identity holds in $\cV^{-1}_\#$.
                \item[(R4)] Let $\bv \in \bcW^1_{t, \perp}$. Then the identity $(\divG \bv)(0) = \divG (\bv(0))$ holds in $H^{-1}_\#$. If $\G$ is of class $C^3$ and $\bv \in \bcW^2_t$, then this identity holds in $L^2_\#$.
            \end{enumerate}
            In particular, $\divG$ maps $\bcW^1_{t, \perp}$ boundedly into $\cW^0_\#$ and, if $\G$ is of class $C^3$, it maps $\bcW^2_t$ boundedly into $\cW^1_\#$.
        \end{lemma}
        \begin{proof}
                        In (R1) and (R2), we test the identities in $H^{-2}_\#$ with $\phi \in H^2_\#$ and those in $H^{-1}_\#$ with $\eta \in H^1_\#$. Recall that $\nablaG \phi \in \bH^1_{t, \perp}$ and $\nablaG \eta \in \bL^2_{t, \perp}$.

            (R1) By the duality definitions \eqref{eq:grad-div},
            \begin{equation*}
                \langle \divG \nablaG u, \phi \rangle_{-2,2} = - \langle \nablaG u, \nablaG \phi \rangle_{-1,1} = (u, \divG \nablaG \phi) = (u, \DeltaG \phi),
            \end{equation*}
            which is the ultraweak Laplace--Beltrami operator. For $u \in H^1_\#$, the same computation with $\eta$ in place of $\phi$ gives $\langle \divG \nablaG u, \eta \rangle_{-1,1} = - (\nablaG u, \nablaG \eta)$. The right-hand side is the weak Laplace--Beltrami operator. For $u \in H^2_\#$, the latter equals $(\DeltaG u, \eta)$ by the divergence theorem.

            (R2) By definition and identity \eqref{eq:stokes-iden-weak},
            \begin{equation*}
                \begin{split}
                    \langle \divG \DeltaS \bv, \phi \rangle_{-2,2} & = - \langle \DeltaS \bv, \nablaG \phi \rangle_{-1,1} = (\DG \bv, \DG \nablaG \phi) \\
                    & = (\divG \bv, \DeltaG \phi) - (\bR \bv, \nablaG \phi).
                \end{split}
            \end{equation*}
            The right-hand side is $\langle \DeltaG \divG \bv + \divG (\bR \bv), \phi \rangle_{-2,2}$, with $\DeltaG$ understood ultraweakly on $\divG \bv \in L^2_\#$. For $\bv \in \bH^2_t$, identity \eqref{eq:stokes-iden-stronger} likewise gives
            \begin{equation*}
                \langle \divG \DeltaS \bv, \eta \rangle_{-1,1} = - (\DeltaS \bv, \nablaG \eta) = - (\nablaG \divG \bv + \bR \bv, \nablaG \eta).
            \end{equation*}
            The right-hand side is $\langle \DeltaG \divG \bv + \divG (\bR \bv), \eta \rangle_{-1,1}$, with $\DeltaG$ understood weakly on $\divG \bv \in H^1_\#$.

            (R3) Since $\G$ is stationary, $\divG$ acts pointwise in time, and bounded linear operators commute with Bochner integrals. Hence, for all $\chi \in C^\infty_c(I)$,
            \begin{equation*}
                \int_I \divG \bv \, \chi' \dt = \divG \int_I \bv \, \chi' \dt = - \divG \int_I \pa_t \bv \, \chi \dt = - \int_I \divG \pa_t \bv \, \chi \dt,
            \end{equation*}
            where $\divG : \bH^{-1}_{t, \perp} \to H^{-2}_\#$ (resp.\ $\divG : \bL^2_t \to H^{-1}_\#$) is bounded. Thus, $\pa_t \divG \bv = \divG \pa_t \bv \in \cV^{-2}_\#$ (resp.\ $\cV^{-1}_\#$). Since also $\divG \bv \in \cV^0_\#$ (resp.\ $\cV^1_\#$), we infer that $\divG \bv \in \cW^0_\#$ (resp.\ $\cW^1_\#$).

            (R4) By Remark~\ref{rmk:lm-vector} (vector-valued counterparts of Lions--Magenes results), $\bv \in C^0([0,T]; \bL^2_{t, \perp})$ (resp.\ $C^0([0,T]; \bH^1_t)$). Since $\divG : \bL^2_t \to H^{-1}_\#$ (resp.\ $\bH^1_t \to L^2_\#$) is bounded, the map $t \mapsto \divG (\bv(t))$ is continuous with values in $H^{-1}_\#$ (resp.\ $L^2_\#$). Moreover, it coincides a.e.\ in $I$ with $\divG \bv$. This map is therefore the continuous representative of $\divG \bv \in \cW^0_\#$ (resp.\ $\cW^1_\#$) from Proposition~\ref{prop:lm-cW} (Lions--Magenes). Evaluating at $t = 0$ yields (R4).

            Finally, the bounds of $\|\divG \bv\|_{\cW^0_\#}$ and $\|\divG \bv\|_{\cW^1_\#}$ follow from (R3), (R4), and the boundedness of $\divG$.
        \end{proof}
        With the rules (R1)--(R4) at hand, the factorization is the formal computation of Section~\ref{subsec:reform-of-stokes-type-pbs}, now with every step justified.
        \begin{lemma}[factorization] \label{lem:factorization}
            Let $\G$ be of class $C^2$ for TS and TTS, and of class $C^3$ with $d=2$ for TTNS. Then, in each of the three settings of Table~\ref{tab:settings}, $\cB = \cD \circ \cA$.
        \end{lemma}
        \begin{proof}
            The formal computation following the composition \eqref{eq:DA-formal} uses only the linearity of $\divG$ and the rules (R1)--(R4). By Lemma~\ref{lem:calculus} (calculus rules), these rules hold for the unknowns in $\sX$, in the spaces in which the corresponding components of $\cB x$ are posed. For TTNS, the convective term appears identically in $\cD(\cA x)$ and $\cB x$, and $\divG (\bu \cdot \nablaG \bu) \in \cV^{-1}_\#$ because $\bu \cdot \nablaG \bu \in \bcV^0_t$.
        \end{proof}

    \subsection{Main results} \label{subsec:main-results}
        The main result is now a direct consequence of Proposition~\ref{prop:transf-of-well-pos} (transfer of well-posedness), Lemma~\ref{lem:data-maps} (data maps), and Lemma~\ref{lem:factorization} (factorization). Theorem~\ref{thm:reformulations} does not assume that the primal problems are well-posed: it shows that each reformulation is well-posed if and only if the corresponding primal problem is.
        \begin{theorem}[equivalence of reformulations] \label{thm:reformulations}
            Let $\G$ be of class $C^2$ for TS and TTS, and let $\G$ be of class $C^3$ and $d = 2$ for TTNS. Let $\cA$ and $\cB$ be the primal and reformulated operators of TS, TTS, or TTNS, with the spaces of Table~\ref{tab:settings}. Then $\cB$ is a homeomorphism if and only if $\cA$ is. In this case, for TS and TTS,
            \begin{equation} \label{eq:stokes-lp-reform-bounds}
                \|\cB\| \leq C \|\cA\|, \qquad \|\cB^{-1}\| \leq C \|\cA^{-1}\|,
            \end{equation}
            where $C$ depends only on $\G$ and $\mu$ (and not on $T$). Moreover, for all data $z$, the solution of $\cB x = z$ is the solution of the primal problem with the same momentum (and initial) data and with the divergence datum $f$ from \eqref{eq:f-from-z}, i.e., the solution of a Poisson problem for TS and of a heat problem for TTS and TTNS.
        \end{theorem}
        \begin{proof}
            Lemmas~\ref{lem:data-maps} (data maps) and~\ref{lem:factorization} (factorization) verify the hypotheses of Proposition~\ref{prop:transf-of-well-pos} (transfer of well-posedness). Its statements (ii) and (iii) yield the assertions with $C := \max\{\|\cD\|, \|\cD^{-1}\|\}$. The characterization of the solution is the solution formula \eqref{eq:reform-as-primal} of Lemma~\ref{lem:triangular} (triangular data maps).
        \end{proof}

        \textit{Well-posedness of the primal problems}. Theorem~\ref{thm:reformulations} (equivalence of reformulations) transfers well-posedness, but does not establish it. For TS, well-posedness in the energy setting is known \cite[Theorem 4.3]{JankuhnOlshanskiiReusken2018}; the proof there is given for $d = 2$ but is valid for $d \geq 2$. For TTS and TTNS, the data maps require nonvanishing divergence data in $\cW^0_\#$ and $\cW^1_\#$, respectively, which is not a standard setting. 
        \begin{proposition}[primal problems; {\cite[Theorem 4.3]{JankuhnOlshanskiiReusken2018}}, \cite{NochettoShakipov2026-reform-part-ii}] \label{prop:primal}
            Let $\G$ be of class $C^2$. Then the primal operators $\cA$ of TS and TTS are linear homeomorphisms, and $\|\cA\|$ and $\|\cA^{-1}\|$ are bounded by constants depending only on $\G$ and $\mu$ (and not on $T$). If, in addition, $\G$ is of class $C^3$ and $d = 2$, then the primal operator $\cA$ of TTNS is a homeomorphism.
        \end{proposition}
        \begin{corollary}[well-posed reformulations] \label{cor:well-posed}
            Under the assumptions of Theorem~\ref{thm:reformulations} (equivalence of reformulations), the reformulated operators $\cB$ of TS, TTS, and TTNS are homeomorphisms. For TS and TTS, $\|\cB\|$ and $\|\cB^{-1}\|$ are bounded by constants depending only on $\G$ and $\mu$ (and not on $T$).
        \end{corollary}
        \begin{proof}
            Combine Theorem~\ref{thm:reformulations} (equivalence of reformulations) with Proposition~\ref{prop:primal} (primal problems).
        \end{proof}
        \begin{remark}[original initial condition] \label{rmk:original-ic}
            Let $\alpha_1 = 1$, and let the datum of the initial Poisson equation vanish, i.e., $z = (\bg_0, \bg, 0, g)$. Then the initial Poisson equation $-\DeltaG u_0 = 0$ yields $u_0 = 0$, so that the relaxed initial condition reduces to $\bu(0) = \bg_0$. Thus, the reformulation \eqref{eq:tts-ttns-reform-gen-data} from the introduction is the special case $g_0 = 0$ of $\cB x = z$. By Corollary~\ref{cor:well-posed} (well-posed reformulations), this reformulation is well-posed for all data $(\bg_0, \bg, g)$ in $\bL^2_{t,\perp} \times \bcV^{-1}_{t,\perp} \times \cV^{-2}_\#$ for TTS, and in $\bH^1_t \times \bcV^0_t \times \cV^{-1}_\#$ for TTNS. In particular, these data need not satisfy any compatibility condition. By Theorem~\ref{thm:reformulations} (equivalence of reformulations), its solution solves the primal problem \eqref{eq:tts-ttns} with the divergence datum $f$ from \eqref{eq:f-from-z}. This datum satisfies $f(0) = \divG \bg_0$, which is precisely the compatibility condition of \eqref{eq:tts-ttns}.
        \end{remark}
        \begin{remark}[equal-order formulations] \label{rmk:equal-order}
            The TS and TTS reformulations test the pressure Poisson equation with $H^2_\#$-functions. If its datum satisfies $g \in H^{-1}_\#$ (resp.\ $g \in \cV^{-1}_\#$), then $H^1_\#$-functions suffice. Indeed, with $\Lambda = - \DeltaG$, the pressure Poisson equation reads $\Lambda u = g - 2 \mu \, \divG (\bR \bu)$. Its right-hand side lies in $H^{-1}_\#$, because $\bR \bu \in \bL^2_t$. The weak operator $\Lambda : H^1_\# \to H^{-1}_\#$ is invertible by Section~\ref{subsec:equiv-norms}, and it is a restriction of the ultraweak operator $\Lambda : L^2_\# \to H^{-2}_\#$. Hence, $u \in H^1_\#$ (resp.\ $u \in \cV^1_\#$). For TS, the pressure Poisson equation thus holds in the form
            \begin{equation} \label{eq:stokes-lp-reform-H1}
                (\nablaG u, \nablaG \phi) - 2 \mu (\bR \bu, \nablaG \phi) = \langle g, \phi \rangle_{-1,1}, \qquad \forall \phi \in H^1_\#,
            \end{equation}
            and the same holds for TTS a.e.\ in time. By Corollary~\ref{cor:well-posed} (well-posed reformulations) and the bounded inverse theorem, the resulting equal-order operators are linear homeomorphisms between $H^1$-based spaces. For TS, for instance, the operator maps $\bH^1_{t, \perp} \times H^1_\#$ onto $\bH^{-1}_{t, \perp} \times H^{-1}_\#$. Hence, $H^1$-conforming finite elements suffice for both the velocity and the pressure, and conforming discretizations do not require a discrete inf-sup condition; see \cite{NochettoShakipov2025}, where the corresponding discrete analysis is carried out for TS with a modified viscous operator.
        \end{remark}

\section{Conclusions} \label{sec:conclusion}
    We have developed a unified approach to reformulations of the tangent Stokes, transient tangent Stokes, and transient tangent Navier--Stokes equations posed on closed hypersurfaces, which circumvent the saddle-point structure of their formulations in the primitive variables. This approach hinges on a simple observation: the reformulation operator factors as $\cB = \cD \circ \cA$, where $\cA$ is the primal operator and $\cD$ is a linear, block-triangular data map whose invertibility reduces to that of a Poisson or a heat operator. Consequently, each reformulation is well-posed if and only if the underlying primal problem is, and solving it with arbitrary data amounts to solving the primal problem with a modified divergence datum. We close with three remarks.

    \textit{Conformity of the test spaces}. As stated, the TS and TTS reformulations \eqref{eq:stokes-lp-reform} and \eqref{eq:ts-reform} test the pressure Poisson equation with $H^2_\#$-functions, which would entail $H^2$-conforming surface elements. Remark~\ref{rmk:equal-order} (equal-order formulations) shows that this is not necessary: if the datum of the pressure Poisson equation lies in $H^{-1}_\#$ (resp.\ $\cV^{-1}_\#$), then the reformulation is well-posed on $H^1$-based spaces, with the pressure Poisson equation tested with $H^1_\#$-functions. The TTNS reformulation \eqref{eq:tns-reform} is $H^1$-conforming in space as stated, which is one of the reasons for favoring the strong formulation of TTNS. We stress that, in all of these $H^1$-based formulations, conforming discretizations do not require a discrete inf-sup condition; see \cite{NochettoShakipov2025} for TS.

    \textit{Other functional settings}. The identity $\cB = \cD \circ \cA$ is a computation that holds whenever its individual steps are justified. Hence, any functional setting in which the primal operator $\cA$ and the operator $L$ are homeomorphisms, e.g., a setting of higher regularity, yields a well-posed reformulation, provided the corresponding theory for the primal problem is available.

    \textit{Stationary tangent Navier--Stokes}. It is natural to ask whether our approach extends to the stationary tangent Navier--Stokes problem. For $d \in \{2,3,4\}$, the stationary tangent Navier--Stokes problem reads: given $(\bf, f) \in \bH^{-1}_{t, \perp} \times L^2_\#$, find $(\bu, u) \in \bH^1_{t, \perp} \times L^2_\#$ such that for all $(\bv, v) \in \bH^1_{t, \perp} \times L^2_\#$,
    \begin{equation} \label{eq:stat-tns}
        \begin{cases}
            2 \mu (\DG \bu, \DG \bv) + \int_\G \bu \cdot \nablaG \bu \cdot \bv \ds - (u, \divG \bv) \kern-0.72em & = \langle \bf, \bv \rangle_{-1,1}, \\
            (\divG \bu, v) \kern-0.72em & = (f, v).
        \end{cases}
    \end{equation}
    Problem \eqref{eq:stat-tns} with $\DG$ replaced by $\nablaG$ was analyzed in \cite[\S 5]{BenavidesNochettoShakipov2025}; we are not aware of an analysis of \eqref{eq:stat-tns} driven by the surface diffusion operator.     Since its test functions are Killing-free, problem~\eqref{eq:stat-tns} ignores the Killing component of the convective term, which may be nonzero if $f \neq 0$; see Remark~\ref{rmk:killing} (Killing fields). For $f = 0$, we expect existence of possibly non-unique solutions, and under a smallness assumption on $(\bf, f)$, we expect well-posedness. For $f \neq 0$, the standard energy bound is lost: testing the momentum equation with $\bu$ produces the cubic term $(\bu \cdot \nablaG \bu, \bu) = - \frac12 (f, |\bu|^2)$, which has no sign. We therefore do not expect the primal operator to be a homeomorphism on the whole data space, in contrast to TTNS, so statement~(ii) of Proposition~\ref{prop:transf-of-well-pos} (transfer of well-posedness) does not apply. However, statement (i) of Proposition~\ref{prop:transf-of-well-pos} requires nothing of $\cA$: the computation behind Lemma~\ref{lem:factorization} (factorization) carries over essentially unchanged, with the additional term $(\bu \cdot \nablaG \bu, \nablaG \phi)$ in the pressure Poisson equation, and yields that the solution sets of $\cB x = z$ and $\cA x = \cD^{-1} z$ coincide. Therefore, the reformulation of \eqref{eq:stat-tns} inherits whatever existence, uniqueness, and stability properties the primal problem enjoys, with the data transformed by $\cD^{-1}$.

\section*{Acknowledgments}
    All main results in this work were proven in full in M. Shakipov's dissertation at the University of Maryland, College Park, without the use of AI. The Claude Opus 5.5 model was used to condense and improve the presentation and the proofs in this work. The authors take full responsibility for the accuracy and integrity of the work.

\bibliographystyle{siamplain}
\bibliography{references}

\end{document}